\documentclass[a4paper,12pt]{article}

\usepackage{inputenc}
\usepackage[T1]{fontenc}
\usepackage{amsmath,amsfonts,amssymb}
\usepackage{amsthm}
\allowdisplaybreaks
\usepackage{a4wide}
\usepackage{subcaption}
\usepackage[style=base]{caption}
\usepackage{csquotes}
\usepackage{booktabs} 
\usepackage{hyperref}

\usepackage{tikz}
\usepackage[percent]{overpic}
\graphicspath{{./Graphics/}}

\usepackage{algorithm}
\usepackage{algorithmic}

\algsetup{linenodelimiter=}

\usepackage[bottom]{footmisc}

\usepackage[backend=bibtex, style=numeric-comp, giveninits=true]{biblatex}
\newcommand{\st}{\text{s.t.}}
\newcommand{\ST}{\quad\text{s.t.}\quad}
\newcommand {\R}{\mathbb{R}}
\newcommand {\N}{\mathbb{N}}
\newcommand {\T}{\mathcal{T}}
\newcommand {\Li}{\mathcal{L}}
\newcommand {\C}{\mathcal{C}}
\newcommand{\supp}{\text{supp}}
\newcommand{\eps}{\varepsilon}

\newtheorem{theorem}{Theorem}[section]
\newtheorem{lemma}[theorem]{Lemma}
\newtheorem{corollary}[theorem]{Corollary}

\newtheorem{definition}[theorem]{Definition}
\newtheorem{example}[theorem]{Example}

\usepackage[english]{babel}

\begin{document}




\author{
  Tim Hoheisel$^{\text{a}}$\\
  \texttt{tim.hoheisel@mcgill.ca}
  \and
  Blanca Pablos$^{\text{b}}$\\
  \texttt{blanca.pablos@unibw.de}
  \and
  Aram-Alexandre Pooladian$^{\text{a}}$\\
  \texttt{aram-alexandre.pooladian@mail.mcgill.ca}
  \and
  Alexandra Schwartz$^{\text{c}}$\\
  \texttt{alexandra.schwartz@tu-darmstadt.de}
  \and
  Luke Steverango$^{\text{d}}$\\
  \texttt{13lwps@queensu.ca}
}
\title{A Study of One-Parameter Regularization Methods for Mathematical Programs with Vanishing Constraints}
\date{
    \small{
    $^{\text{a}}$McGill University, Department of Mathematics and Statistics, 805 Sherbrooke Street West, Montreal, Canada, H3A 0B9\\%
    $^{\text{b}}$University of the Armed Forces, Werner-Heisenberg-Weg 39, 85577 Neubiberg, Germany\\%
    $^{\text{c}}$Technical University of Darmstadt, Dolivostra\ss e 15, 64293 Darmstadt, Germany \\%
    $^{\text{d}}$Queen's University, Department of Mathematics and Statistics, 48 University Ave., Kingston, Ontario, Canada, K7L 3N6 \\}%
    \today
}

\maketitle

\begin{abstract} 
    Mathematical programs with vanishing constraints (MPVCs) are a class of nonlinear optimization problems with applications  to various engineering problems such as truss topology design and robot motion planning.
    MPVCs are difficult problems from both a theoretical and numerical perspective:  the combinatorial nature of the vanishing constraints often  prevents  standard constraint qualifications and  optimality conditions from being attained; moreover,  the feasible set  is inherently nonconvex, and  often has no interior around points of interest. 
    In this paper, we therefore study and compare four regularization methods for the numerical solution of MPVCS. Each method depends on a single regularization parameter, which is used to embed the original MPVC into a sequence of standard nonlinear programs.
	Convergence results for these methods based on both exact and approximate stationary of the subproblems are established under weak assumptions.
	The improved regularity of the subproblems is studied by providing sufficient conditions for the existence of KKT multipliers.
    Numerical experiments, based on applications in truss topology design and an optimal control problem from aerothermodynamics, complement the theoretical analysis and comparison of the regularization methods.
	The computational results highlight the benefit of using  regularization  over applying a standard solver directly, and they allow us to identify two promising regularization schemes.
\end{abstract}




\section{Introduction}\label{sec:intro}

We consider \textit{mathematical programs with vanishing constraints (MPVC)}, which are constrained optimization problems of the form 
\begin{equation}\label{eq:MPVC}
    \begin{array}{ll}
       \min & f(x) \\
       \text{s.t.}    & g_i(x) \leq 0 \quad (i=1,\ldots,m),\\
       & h_i(x) = 0 \quad (i=1,\ldots,p),\\
       & H_i(x) \geq 0 \quad (i=1, \ldots, l), \\
       & G_i(x)H_i(x)\leq 0 \quad (i=1, \ldots, l).
    \end{array}
\end{equation} 
We assume throughout the paper that the functions $f,g_i,h_i,H_i,G_i:\R^n \to \R$ are continuously  differentiable.
MPVCs were introduced to the mathematical community in a seminal paper by Achtziger and Kanzow \cite{AcK08}, where they were extracted as a mathematical model for stress constraints in optimal topology design of mechanical structures.
Other applications of MPVCs comprise e.g. robot motion planning  and  mixed-integer nonlinear optimal control problems (MIOCPs) \cite{Jung2013}.
The theoretical foundations, i.e.  optimality, stationarity, criticality and  constraint qualifications have been established in the literature by Hoheisel et al. \cite{Hoh09, HoK09, HoK08, HoK07, HKO10} and other authors \cite{DSS12, IzP09}.
Numerical schemes based on smoothing and relaxation were studied  by  Hoheisel et al. \cite{AHK13, AHK12, HKS12local, HKS12new}.
Izmailov et al. \cite{IzS09}, and Dussault et al \cite{DHM19} also study regularization methods.
Izmailov et al. also study  Newton-type methods \cite{IzP09} for MPVCs.
The recent paper by Benko and Gfrerer \cite{BeG17} establishes an SQP method for MPVCs based on their very own {\em Q-stationarity}. 

In this paper we focus on regularization methods for the  solution of MPVCs, that depend on a single parameter $t>0$, since these have proven to be  simple and robust numerical approaches to  MPVCs  \cite{AHK13, AHK12, HKS12local, HKS12new,IzS09}.
The general idea of these methods is to consider a regularization  $X(t)\subset \R^n$ of the feasible set  of \eqref{eq:MPVC}  which is less degenerate in terms of constraint qualifications and existence of interior points, and such that $X(0)$ coincides with the original feasible set.
Given a sequence $\{t_k\}\downarrow 0$, the resulting numerical strategy is  to asymptotically approximate critical points of \eqref{eq:MPVC} via critical points of $f$ over $X(t_k)$.
The concrete regularization methods studied in this paper are:
\begin{itemize}
    \item the global regularization method \cite{AHK12, IzS09};
    \item the local regularization method \cite{HKS12local};
    \item the $L$-shaped regularization method \cite{HKS12new};
    \item the nonsmooth regularization method.
\end{itemize}
The latter has not been previously considered for MPVCs  but, like the others, cf. \cite{HKS13}, is  inspired by an analogous approach to {\em mathematical programs with complementarity constraints (MPCCs)}  due to Kadrani et al. \cite{KDB09}.    

This paper is the first systematic theoretical and numerical comparison of regularization methods for MPVCs.
For the regularization methods under consideration we present convergence results under weak constraint qualifications,   comprising both exact and inexact notions of stationarity.
Moreover, we study local regularity properties of the regularized feasible sets, which also establish the existence of KKT multipliers for the regularized problems and thus prove that the regularized problems are less degenerate.
A numerical comparison, illustrating the  benefit of regularization, is provided based on two instances from truss topology optimization  and an example from aerothermodynamics.

Inexact convergence result have previously not been available for MPVCs, and the exact results in the improved given form are, unless otherwise stated, also new.
The same holds for our study of the local regularity properties of the regularized feasible sets.
Numerical tests of regularization methods for MPVCs on truss topology problems are well established \cite{AHK13, AHK12, DHM19}.
After all this was the motivating instance of this  problem class.
The aerothermodynamics example is new and MPVCs have, to the best of our knowledge, never been used in this context before.

Finally, since the regularization schemes in question have their counterpart for MPCCs \cite{HKS13}, we shed light on the question whether the respective approaches have similar theoretical  and numerical properties as for MPCCs.

The organization of this paper is as follows:
In Section~\ref{sec:preliminaries} we provide some background material from nonlinear programming and MPVC theory, in particular different notions of stationarity and constraint qualifications. 
Exact and inexact convergence properties of four different  regularization methods are investigated in Section~\ref{sec:convergence}.
In Section \ref{sec:numerical}, we present a numerical comparison of the regularization methods studied on problems from truss topology design and aerothermodynamics.
We conclude with some final remarks in Section~\ref{sec:final}.

Most of the notation used is standard:
For a differentiable function $f:\R^n \to \R$, $\nabla f(x)\in \R^n$ denotes the gradient of $f$ at $x$, which is understood as a column vector.
For a vector $x\in\R^n$ we denote by
$
   \supp(x):=\big\{i\in \{1,\dots,n\} \mid x_i\neq 0 \big\}
$
the support of this vector.
Given a subset $I\subset \{1,\dots,n\}$, we use the abbreviation $x_I:=(x_i)_{i\in I}\in \R^{|I|}$. 
\section{Preliminaries}\label{sec:preliminaries}

In this section, we recall necessary background knowledge for both standard nonlinear programs and MPVCs, with a strong focus on optimality conditions and constraint qualifications.

\subsection{Constraint Qualifications for Standard Nonlinear Programs}\label{ssec:NLP}

The central idea behind a regularization method is to replace the difficult MPVC by a sequence of (hopefully simpler) standard nonlinear programs.
Thus, we begin by recalling some constraint qualifications for this problem class.
Consider the following \emph{nonlinear program (NLP)}
\begin{equation}\label{eq:NLP}
    \begin{array}{rcl}
         \min f(x) & \st & g_i(x) \leq 0 \quad (i=1,\ldots,m),\\
         && h_i(x)=0 \quad (i=1,\ldots,p)
    \end{array}
\end{equation}
and let $Z$ denote the set of feasible points of \eqref{eq:NLP}.
For an arbitrary $x^*\in Z$ we denote the \emph{set of active inequality constraints} by
$$
   I_g(x^*) := \{i \mid g_i(x^*) = 0\}.
$$
Furthermore, the \emph{(Bouligand) tangent cone} of $Z$ at $x^* \in Z$ is defined as
\begin{equation*}
    \T_Z (x^*) := \left\{ d \in \R^n \mid \exists \{ x^k\} \subset Z, \exists \{ \tau_k \} \geq 0 \text{ such that }
    x^k  \to x^* \text{ and } \tau_k (x^k - x^*) \to d \right\}, 
\end{equation*}
and the \emph{linearized cone} of $Z$ at $x^* \in Z$ is given by
\begin{equation*}
   \Li_Z (x^*) := \big\{ d \in \R^n \mid \nabla g_i(x^*)^T d \le 0 \; (i \in I_g(x^*)),
   \; \nabla h_i (x^*)^T d = 0 \; (i = 1, \ldots, p) \big\}.
\end{equation*}
Furthermore, the \emph{polar cone} to an arbitrary cone $\C \subseteq \R^n$ is defined as
$$
   \C^\circ := \{s \in \R^n \mid s^T d \leq 0 \;  \forall d \in \C \}.
$$

One of the constraint qualifications we are going to state uses positive linear independence of vectors.
We therefore first recall the definition thereof.

\begin{definition}\label{def:PositiveLinearDependence}
    A set of vectors
    $$
       a_i \; (i \in I_1) \text{ and } b_i \; (i \in I_2)
    $$
    is said to be \emph{positively linearly dependent} if there exist scalars 
    $\alpha_i \ (i \in I_1)$ and $\beta_i \ (i \in I_2) $, not all of them being 
    zero, with $\alpha_i \geq 0$ for all $i\in I_1$ and
    $$
       \sum_{i \in I_1}{\alpha_i a_i} + \sum_{i\in I_2}{\beta_i b_i} = 0.
    $$
    Otherwise, we say that these vectors are \emph{positively linearly independent}.
\end{definition}

With these definitions, we are now able to define \emph{constraint qualifications (CQ)} for NLPs.

\begin{definition}
    A feasible point $x^*$ for \eqref{eq:NLP} is said to satisfy the
    \begin{itemize}
       \item[(a)] \emph{linear independence CQ (LICQ)}, if the gradients
        $$
            \nabla g_i(x^*) \; (i \in I_g(x^*)), \quad \nabla h_i(x^*) \; (i=1, \ldots,p)
        $$
        are linearly independent;
      
        \item[(b)] \emph{Mangasarian-Fromovitz CQ (MFCQ)}, if the gradients
        $$
             \nabla g_i(x^*) \; (i \in I_g(x^*)) \text{ and } \nabla h_i(x^*) \; (i=1, \ldots,p)
        $$
        are positively linearly independent;
    
        \item[(c)] \emph{constant rank CQ (CRCQ)}, if there exists a neighborhood $ N(x^*) $ of $ x^* $, such that for all subsets $ I_1 \subseteq I_g(x^*) $ and $ I_2 \subseteq \{ 1, \ldots, p \} $, the gradient vectors
        $$
            \nabla g_i(x) \; (i \in I_1),  \quad \nabla h_i(x) \; (i \in I_2)
        $$
        have constant rank for all  $ x \in N(x^*) $ (which depends on $I_1, I_2$);
        
        \item[(d)] \emph{constant positive linear dependence CQ} (CPLD), if there exists a neighborhood $ N(x^*) $ of $ x^* $, such that for any subsets $ I_1 \subseteq I_g(x^*) $ and $ I_2 \subseteq \{ 1, \ldots, p \} $, for which the gradients
        \begin{equation*}
            \nabla g_i(x) \; (i \in I_1) \text{ and }
            \nabla h_i(x) \; (i \in I_2) 
        \end{equation*}
        are positively linearly dependent in $x^*$, they remain linearly dependent on $N(x^*) $;
        
        \item[(e)] \emph{Abadie CQ (ACQ)} if $ \T_Z(x^*) = \Li_Z(x^*) $;
        
        \item[(f)] \emph{Guignard CQ (GCQ)} if $ \T_Z(x^*)^\circ = \Li_Z(x^*)^\circ $.
    \end{itemize}
\end{definition}

The following relations hold between these constraint qualifications:
\begin{center}
    \begin{tikzpicture}[> = stealth]
        \node at (0,0) (licq) {LICQ};
		\node at (2,0.5) (mfcq)  {MFCQ};
		\node at (2,-0.5) (crcq)  {CRCQ};
		\node at (4,0) (cpld)  {CPLD};
		\node at (6,0) (acq)  {ACQ};
		\node at (8,0) (gcq) {GCQ};

		\draw[->, double distance=1.5pt, thick] (licq) -- (mfcq);
		\draw[->, double distance=1.5pt, thick] (licq) -- (crcq);
		\draw[->, double distance=1.5pt, thick] (mfcq) -- (cpld);
		\draw[->, double distance=1.5pt, thick] (crcq) -- (cpld);
		\draw[->, double distance=1.5pt, thick] (cpld) -- (acq);
		\draw[->, double distance=1.5pt, thick] (acq) -- (gcq);
    \end{tikzpicture}
\end{center}
It was proven in \cite{AMS05} that CPLD implies ACQ.
All other implications follow directly from the definitions.
It is well known that for  every local minimizer $x^*$ of \eqref{eq:NLP}, in which GCQ holds, there exist multipliers $\lambda \in \R^m $ and $\mu \in \R^p$ such that 
$$
   0 = \nabla f(x^*) + \sum_{i=1}^m{\lambda_i \nabla g_i(x^*)} + \sum_{i=1}^p{\mu_i \nabla h_i(x^*)}
$$
with $\supp(\lambda) \subseteq I_g(x^*)$ and $ \lambda \ge 0$.
In this situation, by slight abuse of terminology, we  will refer to both $x^*$ and the triple $ (x^*, \lambda, \mu) $ as  a \emph{KKT point} of  \eqref{eq:NLP}.

From a computational perspective, one cannot expect to obtain exact KKT points of a given NLP.
Hence, the following generalized notion will play a role in our analysis.

\begin{definition}
    Let $x^{*} \in \R^n$ and $\eps \geq 0$.
    If there exist $\lambda \in \R^m$ and $\mu \in \R^p$ such that 
    $$
        \left\| \nabla f(x^{*}) + \sum_{i=1}^m \lambda_i\nabla g_i(x^*) + \sum_{i=1}^p \mu_i \nabla  h_i(x^{*}) \right\|_\infty \leq \eps,
    $$
    with $g_i(x^{*}) \leq \eps$, $\lambda_i \geq -\eps$, $|g_i(x^{*})\lambda_i| \leq \eps$ and $|h_i(x^{*})| \leq \eps$ (for all $i = 1, \ldots, m$ and $i = 1,\ldots, p$, respectively), then $x^{*}$ is called an \emph{$\eps$-stationary point} of \eqref{eq:NLP}.
\end{definition}

\subsection{Stationary Points for MPVCs}\label{ssec:MPVCstationary}

While the KKT conditions are the single most important necessary optimality criterion for NLPs, there are several stationarity concepts in use when it comes to MPVCs.
The reason for this is that MPVCs violate NLP constraint qualifications in many important and relevant situations, see e.g. \cite{Hoh09}.
They thus require tailored optimality conditions and CQs.

In order to state MPVC-tailored stationarity conditions, we need the following index sets:
Let $x^*$ be an arbitrary feasible point of \eqref{eq:MPVC} and  $ I_g = \{ i \ | \ g_i(x^*) = 0 \} $ be defined as before.
Additional index sets corresponding to the vanishing constraints are defined as
\begin{equation*}
   I_+  :=  \displaystyle{\big\{ i \, \big| \, H_i(x^*) > 0 \big\}},\quad 
   I_0  :=  \displaystyle{\big\{ i \, \big| \, H_i(x^*) = 0 \big\}.}  
\end{equation*}
Furthermore, we divide the index set $ I_+ $ into the following subsets:
\begin{eqnarray*}
   I_{+0} & := & \displaystyle{\big\{ i \, \big| \, H_i(x^*) > 0, G_i(x^*) 
   = 0 \big\},} \\
   I_{+-} & := & \displaystyle{\big\{ i \, \big| \, H_i(x^*) > 0, G_i(x^*) 
   < 0 \big\}.}
\end{eqnarray*}
Similarly, we partition the set $ I_0 $ in the following way:
\begin{eqnarray*}
   I_{0+} & := & \displaystyle{\big\{ i \, \big| \, H_i(x^*) = 0, 
   G_i(x^*) > 0 \big\},} \\
   I_{00} & := & \displaystyle{\big\{ i \, \big| \, H_i(x^*) = 0, 
   G_i(x^*) = 0 \big\},} \\
   I_{0-} & := & \displaystyle{\big\{ i \, \big| \, H_i(x^*) = 0, 
   G_i(x^*) < 0 \big\}.}
\end{eqnarray*}
Note that the first subscript indicates the sign of $ H_i(x^*) $, whereas the second subscript stands for the sign of $ G_i(x^*) $.
We would also like to point out that the above index sets substantially depend on the chosen point $x^*$.
Throughout this section, it will always be clear from the context which point these index sets refer to. 

\begin{definition}\label{def:weakstat}
    Let $x^*$ be feasible for the MPVC \eqref{eq:MPVC}.
    Then $x^*$ is called 
    \begin{itemize}
       \item[(a)] \emph{weakly stationary}, if there exist multipliers  $ \lambda \in \R^m, \mu \in \R^p, \eta^H, \eta^G \in \R^l $ such that
        \begin{eqnarray*}
            && \nabla f(x^*) + \sum_{i=1}^m \lambda_i \nabla g_i(x^*) +
            \sum_{i=1}^p \mu_i \nabla h_i(x^*) - \sum_{i=1}^l \eta_i^H
            \nabla H_i(x^*) + \sum_{i=1}^l \eta_i^G \nabla G_i(x^*) = 0,  \\        
            && \lambda_i \ge 0 \; (i \in I_g),  \quad 
            \lambda_i = 0 \ (i \notin I_g), \\
            && \eta_i^H = 0 \; (i \in I_+), \quad
            \eta_i^H \ge 0 \;  (i \in I_{0-}), \quad
            \eta_i^H \text{ free }  (i \in I_{0+}\cup I_{00}), \\
            && \eta_i^G = 0 \, (i \in I_{+-} \cup I_{0-}\cup I_{0+}), \quad
            \eta_i^G \ge 0 \, (i \in I_{+0}\cup I_{00} );
        \end{eqnarray*}
        
        \item[(b)] \emph{T-stationary} if   $ x^* $ is weakly stationary and  
          $\eta_i^G \eta_i^H \leq 0  $ for all $i \in I_{00}$;
          
        \item[(c)] \emph{M-stationary} if $ x^* $ is weakly stationary and 
          $\eta_i^G \eta_i^H = 0  $ for all $i \in I_{00}$;
          
        \item[(d)] \emph{S-stationary} if $x^*$ is weakly stationary and 
          $\eta_i^H \geq0, \eta_i^G=0$ for all $i \in I_{00}$.
    \end{itemize}
\end{definition}

By slight abuse of terminology, if $x^*$ is a weakly/T-/M-S-stationary point with the multipliers  $(\lambda,\mu,\eta^G,\eta^H)$, then we also call the whole quintuple $(x^*,\lambda,\mu,\eta^G,\eta^H)$ weakly/T-/M-S-stationary.
Obviously, the following implications hold for these stationarity concepts:
\begin{center}
    \begin{tikzpicture}[> = stealth]
        \node at (0,0) (s-stat) {S-stationarity};
		\node at (3.5,0) (m-stat)  {M-stationarity};
		\node at (7,0) (t-stat)  {T-stationarity};
		\node at (10.5,0) (w-stat)  {weak stationarity};

		\draw[->, double distance=1.5pt, thick] (s-stat) -- (m-stat);
		\draw[->, double distance=1.5pt, thick] (m-stat) -- (t-stat);
		\draw[->, double distance=1.5pt, thick] (t-stat) -- (w-stat);
    \end{tikzpicture}
\end{center}
The only difference between these four stationarity concepts lies in the conditions on the multipliers corresponding to the bi-active set $ I_{00} $. 
These conditions are illustrated in Figure \ref{fig:statconcepts}.
Hence, if the bi-active set is empty, all four stationary concepts coincide. 
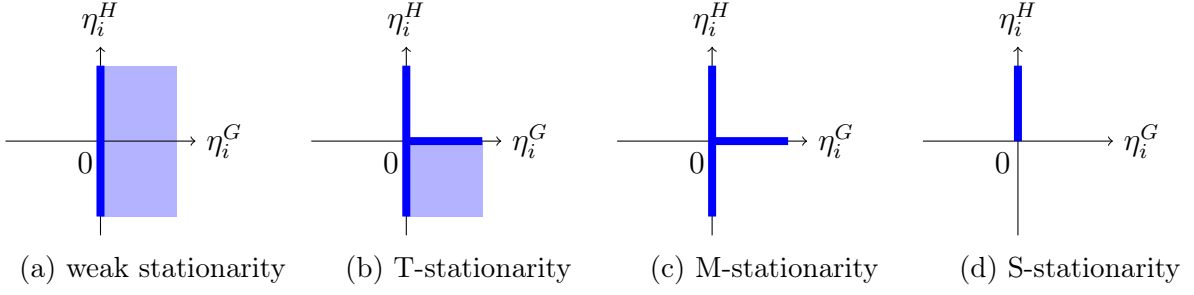
\begin{figure}
    \centering
    \begin{subfigure}{0.24\textwidth}
        \begin{tikzpicture}
            \fill[color = blue!30] (0, -1) -- (1,-1) -- (1, 1) -- (0, 1) -- cycle;
            \draw[-] (-1.25,0) -- (-0.2,0) node[below] {$0$} -- (0,0);
     	    \draw[->] (0,0) -- (1.25,0) node[right]  {$\eta^G_i$};
     	    \draw[->] (0,-1.25)  -- (0,1.25)  node[above]  {$\eta^H_i$};
     	    \draw[color = blue, line width = 3pt] (0, 1) -- (0, -1);
        \end{tikzpicture}
        \caption{weak stationarity}
    \end{subfigure}
    \begin{subfigure}{0.24\textwidth}
        \begin{tikzpicture}
         	\fill[color = blue!30] (0, -1) -- (1,-1) -- (1, 0) -- (0, 0) -- cycle;
            \draw[-] (-1.25,0) -- (-0.2,0) node[below] {$0$} -- (0,0);
         	\draw[->] (0,0) -- (1.25,0) node[right]  {$\eta^G_i$};
         	\draw[->] (0,-1.25)  -- (0,1.25)  node[above]  {$\eta^H_i$};
         	\draw[color = blue, line width = 3pt] (0, 1) -- (0, -1);
         	\draw[color = blue, line width = 3pt] (0, 0) -- (1, 0);
        \end{tikzpicture}
        \caption{T-stationarity}
    \end{subfigure}
    \begin{subfigure}{0.24\textwidth}
        \begin{tikzpicture}
           \draw[-] (-1.25,0) -- (-0.2,0) node[below] {$0$} -- (0,0);
            \draw[->] (0,0) -- (1.25,0) node[right]  {$\eta^G_i$};
            \draw[->] (0,-1.25)  -- (0,1.25)  node[above]  {$\eta^H_i$};
            \draw[color = blue, line width = 3pt] (0, 1) -- (0, -1);
            \draw[color = blue, line width = 3pt] (0, 0) -- (1, 0);
        \end{tikzpicture}
        \caption{M-stationarity}
    \end{subfigure}
    \begin{subfigure}{0.24\textwidth}
        \begin{tikzpicture}
            \draw[-] (-1.25,0) -- (-0.2,0) node[below] {$0$} -- (0,0);
            \draw[->] (0,0) -- (1.25,0) node[right]  {$\eta^G_i$};
            \draw[->] (0,-1.25)  -- (0,1.25)  node[above]  {$\eta^H_i$};
            \draw[color = blue, line width = 3pt] (0, 0) -- (0, 1);
        \end{tikzpicture}
        \caption{S-stationarity}
    \end{subfigure}
    \caption{Geometric illustration of weak, T-, M-, and S-stationarity for an 
    index $i \in I_{00}$}\label{fig:statconcepts}
\end{figure}

The notion of weak stationarity for MPVCs was introduced in \cite{IzS09}, whereas M-statio\-narity for MPVCs is due to \cite{HoK08} and S-stationarity, which is equivalent to the KKT conditions of \eqref{eq:MPVC}, was first mentioned in \cite{AcK08}.
T-stationarity was coined in \cite{DSS12}.
In an MPCC setting, the counterpart of T-stationarity is usually called C-stationarity, cf.\ \cite{ScS00}.

\subsection{MPVC-tailored Constraint Qualifications}

As was pointed out above,  most standard constraint qualifications are violated by the vanishing constraints.
For this reason, a myriad of MPVC-tailored constraint qualifications have been developed in the past, see e.g. \cite{Hoh09}.
To keep our presentation unified and compact, we confine ourselves to the ones most useful to our study.

\begin{definition}\label{def:cqs}
    A feasible point $x^*$ of the MPVC \eqref{eq:MPVC} is said to satisfy 
    \begin{itemize}
        \item[(a)] \emph{MPVC-linear independence CQ (MPVC-LICQ)}, if the gradients
        \begin{equation*}
            \nabla g_i(x^*) \, (i \in I_g), \;
            \nabla h_i(x^*) \, (i = 1, \ldots, p), \; 
            \nabla G_i(x^*) \, (i \in I_{00} \cup I_{+0}), \;
            \nabla H_i(x^*) \, (i \in I_{0})
        \end{equation*}
        are linearly independent;
          
        \item[(b)] \emph{MPVC-Mangasarian-Fromovitz CQ (MPVC-MFCQ)}, if the gradients 
        \begin{eqnarray*}
           \nabla g_i(x^*)  \, (i \in I_g), \;
            (-\nabla H_i(x^*))  \, (i \in I_{0-}), \;
            \nabla G_i(x^*)  \, (i \in I_{+0} \cup I_{00}) & \text{and}&  \\
            \nabla h_i(x^*) \, (i=1,\dots,p ), \;
            \nabla H_i(x^*) \, (i \in I_{0+} \cup I_{00}) &&
        \end{eqnarray*}
        are positively linearly independent;
        
       \item[(c)] \emph{MPVC-constant rank CQ (MPVC-CRCQ)}, if there is a neighborhood $N(x^*)$ of $x^*$ such that for all subsets $I_1 \subseteq I_g$, $I_2 \subseteq \{1,\ldots,p\}$, $I_3 \subseteq I_{+0}\cup I_{00}$, $I_4 \subseteq I_{0}$, the gradients
       \begin{equation*}
            \nabla g_i(x) \, (i \in I_1), \;
            \nabla h_i(x) \, (i \in I_2), \; 
            \nabla G_i(x) \, (i \in I_3), \;
            \nabla H_i(x) \, (i \in I_4)
        \end{equation*}
       have  constant rank for all $x \in N(x^*)$;
       
       \item[(d)] \emph{MPVC-constant positive linear dependence CQ (MPVC-CPLD)}, if there is a neighborhood $N(x^*)$ of $x^*$ such that for all subsets $ I_1 \subseteq I_g$, $I_2\subseteq I_{0-}$, $I_3\subseteq I_{+0}\cup I_{00}$, $ I_4 \subseteq \{ 1, \ldots, p \} $, $I_5\subseteq I_{0+}\cup I_{00}$, for which the gradients
       \begin{eqnarray*}
           \nabla g_i(x)  \, (i \in I_1), \;
            (-\nabla H_i(x))  \, (i \in I_2), \;
            \nabla G_i(x)  \, (i \in I_3) & \text{and}&  \\
            \nabla h_i(x) \, (i \in I_4 ), \;
            \nabla H_i(x) \, (i \in I_5) &&
        \end{eqnarray*}
        are positively linearly dependent in $x^*$, they remain linearly dependent for all $x \in N(x^*)$.
    \end{itemize}
\end{definition}

In the definition of MPVC-MFCQ and MPVC-CPLD,  we use the word "and" to separate the gradients, for which there are sign constraints in the definition of positive linear dependence, from those without sign constraints.

Apart from those defined above, there exist a number of constraint qualifications tailored to MPVCs such as MPVC-ACQ, a variant of the standard ACQ.
Some of the relations between these constraint qualifications are displayed in the diagram below, see \cite{HKS12local} and the references therein for more information about these constraint qualifications.
\begin{center}
    \begin{tikzpicture}[> = stealth]
        \node at (0,0) (licq) {MPVC-LICQ};
		\node at (4,0.5) (mfcq)  {MPVC-MFCQ};
		\node at (4,-0.5) (crcq)  {MPVC-CRCQ};
		\node at (8,0) (cpld)  {MPVC-CPLD};
		\node at (12,0) (acq)  {MPVC-ACQ};

		\draw[->, double distance=1.5pt, thick] (licq) -- (mfcq);
		\draw[->, double distance=1.5pt, thick] (licq) -- (crcq);
		\draw[->, double distance=1.5pt, thick] (mfcq) -- (cpld);
		\draw[->, double distance=1.5pt, thick] (crcq) -- (cpld);
		\draw[->, double distance=1.5pt, thick] (cpld) -- (acq);
    \end{tikzpicture}
\end{center}
Analogous to the standard case, MPVC-LICQ is the strongest constraint qualification of the five mentioned here and MPVC-ACQ is the weakest.
MPVC-CPLD  relaxes both MPVC-MFCQ and MPVC-CRCQ, whereas it is known that neither MPVC-MFCQ implies MPVC-CRCQ nor vice versa.

\section{Convergence Properties of Regularization Schemes}\label{sec:convergence}

In this section, we discuss the theoretical properties of the four regularization schemes from \cite{Sch01, StU10, KaS13, KDB09} in detail.
All of them were originally introduced for MPCCs and have since been adapted to MPVCs in \cite{AHK12, IzS09, HKS12local, HKS12new}.
The only exception is the regularization scheme from \cite{KDB09}, which is discussed in the context of MPVCs for the first time in this paper.

In previous analysis of these regularization methods, it is usually assumed that one is able to compute an exact KKT point of the regularized NLPs in every iteration.
However, from a numerical point of view, this is not a realistic assumption.
Furthermore, for MPCCs an in-depth analysis \cite{KaS15} revealed that computing only $\eps$-stationary points of the regularized problems has serious effects on the convergence properties of some of these schemes. 

For this reason, we develop convergence results based on $\eps$-stationary points of the regularized problems and compare them to the -- partially preexisting -- convergence results based on exact KKT points. \footnote{Here, for all convergence results, the point $x^*$ in question is always feasible for the underlying MPVC, so that the imposed  constraint qualifications at $x^*$ are well-defined.}

Furthermore, we provide conditions on the MPVC under which the regularized problems satisfy a standard CQ locally.

For clarity and brevity of notation we omit standard inequality and equality constraints in the following proofs.

\subsection{The Global Regularization}\label{ssec:globalRegularization}

The first regularization method we present has been studied in \cite{AHK12, IzS09} and is similar to the method for MPCCs proposed by Scholtes in his seminal paper \cite{Sch01}.
For a given regularization parameter $t> 0$ the regularized problem reads

\begin{minipage}[c]{0.5\textwidth}
    \begin{align} 
        \min \ & f(x) \nonumber\\ 
        \st \ & g_i(x) \leq 0 \  (i = 1,\ldots,m), \nonumber  \\ 
         & h_j(x) = 0 \ (j = 1,\ldots,p), \tag{$R^S(t)$} \nonumber \\ 
         & H_i(x) \geq 0 \ (i = 1,\ldots,l),  \nonumber \\ 
         & G_i(x) H_i(x) \leq t \  (i = 1,\ldots,l).  \nonumber 
    \end{align}
\end{minipage}
\qquad
\begin{minipage}[c]{0.45\textwidth}
    \begin{tikzpicture}[scale = 1.25]
        \fill[blue!30] (-1.5,0) -- (0,0) -- (0,2) -- (-1.5,2) -- cycle;
        \fill[blue!30] (0.25,2) -- plot[domain=0.25:2] (\x,{0.5/\x}) -- (2,0.25) -- (2,0) -- (0,0) -- (0,2) -- cycle;
        \draw[->] (-1.75,0) -- ({sqrt(0.5)},0) node[below]{$\sqrt{t}$} -- (2.5,0) node[below]{$G_i(x)$};
        \draw[->] (0,-0.5) -- (0,{sqrt(0.5)}) node[left]{$\sqrt{t}$} -- (0,2.5) node[left]{$H_i(x)$};
        \draw[, line width = 1.5pt, blue] (-1.5,0) -- (0,0) -- (2,0);
        \draw[, line width = 1.5pt, blue] plot[domain=0.25:2] (\x,{0.5/\x});
        \draw[dashed] ({sqrt(0.5)},0) -- ({sqrt(0.5)},{sqrt(0.5)}) -- (0,{sqrt(0.5)});
        \draw[->] (0.5,1.75) node[right]{$G_i(x)H_i(x)=t$} -- (1/3,3/2);
    \end{tikzpicture}
\end{minipage}

We denote the feasible set of the regularized program $R^S(t)$ by $X^S(t)$.
Due to the structure of the resulting regularized feasible set, we call this approach the \emph{global regularization}.
For a given $t>0$ and $x\in X^S(t)$, we define the index sets
\begin{eqnarray*}
   I_g(x)&:=& \{i \mid g_i(x) = 0\},\\
   I_H(x)&:=&\{i\mid H_i(x)=0\},\\
   I_{GH}(x,t)&:=&\{i\mid G_i(x)H_i(x)=t\}.
\end{eqnarray*}

Then we obtain the following convergence result based on $\eps$-stationary points of the regularized problems.

\begin{theorem}\label{thm:globalConvergenceEpsMFCQ}
    Let $\{t_k\} \downarrow 0$, $\{\eps_k\} \downarrow 0$ with $\eps_k = O(t_k)$ and $\{x^k\}$ a sequence of $\eps_k$-stationary points of $R^S(t_k)$ with $x^k \to x^*$.
    If MPVC-MFCQ holds in $x^*$, then $x^*$ is a T-stationary point of the MPVC \eqref{eq:MPVC}.
\end{theorem}

\begin{proof}
    Since $x^k$ is an $\eps_k$-stationary point of $R^S(t_k)$, there exist multipliers $\nu^k, \delta^k \in \R^l$ such that
    \begin{eqnarray}
        && \left\|\nabla f(x^k) - \sum_{i=1}^l \nu_i^k \nabla H_i(x^k) + \sum_{i=1}^l \big[ \delta_i^k  G_i(x^k)\nabla H_i(x^k) +  \delta_i^k H_i(x^k) \nabla G_i(x^k) \big] \right\|_\infty \leq \eps_k, \label{eq:InexactScholtes} \\
        && \nu_i^k \geq -\eps_k,\quad H_i(x^k) \geq -\eps_k,\quad  |\nu^k_i H_i(x^k)| \leq \eps_k \quad (i=1,\dots,l), \label{eq:InexactSch1} \\
        && \delta_i^k\geq -\eps_k,\quad  G_i(x^k)H_i(x^k) - t_k \leq \eps_k, \quad  |\delta_i^k(G_i(x^k)H_i(x^k) - t_k)| \leq \varepsilon_k\quad (i=1,\dots,l). \label{eq:InexactSch2}
    \end{eqnarray}

    First of all, this implies that the limit $x^*$ is feasible for the MPVC \eqref{eq:MPVC}.
    Furthermore, we can draw some conclusions from this regarding the multipliers:
    For all $i \in I_{+}$ using $H_i(x^*) > 0$ the conditions in \eqref{eq:InexactSch1} imply $\nu_i^k = 0$  for  all $k$ large and thus $\lim_{k \to \infty} \nu_i^k = 0$.
    For all $i \in I_0$ we know at least $\liminf_{k \to \infty} \nu_i^k \geq 0$.

    For all $i \in I_{+-}$ using $G_i(x^*)H_i(x^*) < 0$ together with \eqref{eq:InexactSch2} implies $\lim_{k \to \infty} \delta_i^k = 0$.
    For all $i \notin I_{+-}$ we obtain $\liminf_{k \to \infty} \delta_i^k \geq 0$.

    A case needed frequently later on will be $|\delta^k_i| \to \infty$.
    Based on the previous discussion, this implies $i \notin I_{+-}$ as well as $\delta^k_i \to \infty$.
    In case $\eps_k > 0$ for all $k$ large, we can infer from \eqref{eq:InexactSch2} that
    \begin{equation*}
        \tfrac{1}{\eps_k} |G_i(x^k) H_i(x^k) -t_k| \leq \tfrac{1}{\delta_k} \to 0.
    \end{equation*}
    Then using $\eps_k = O(t_k)$ yields
    \begin{equation}\label{eq:deltaUnboundedScholtes}
        \tfrac{1}{\eps_k} G_i(x^k) H_i(x^k) \not\to 0.
    \end{equation}

    Now define $\eta^{G,k}, \eta^{H,k} \in \R^l$ component-wise as
    \begin{equation*}
        \eta^{G,k}_i := 
        \begin{cases}
           \delta_i^k H_i(x^k), & i \in I_{00} \cup I_{+0}, \\
           0, & i \in I_{0+} \cup I_{0-} \cup I_{+-},
        \end{cases}
        \quad\text{and}\quad
        \eta_i^{H,k} := 
        \begin{cases}
           \nu^k_i - \delta^k_i G_i(x^k), & i \in I_{0}, \\
           \nu^k_i, & i \in I_{+}.
        \end{cases}
    \end{equation*}
    Then we can rewrite (\ref{eq:InexactScholtes}) as
    \begin{align}
        \begin{split}\label{eq:InexactScholtes2}
            \Big\|\nabla f(x^k) & - \sum_{i=1}^l \eta^{H,k}_i \nabla H_i(x^k) + \sum_{i \in I_{+-} \cup I_{+0}}(\delta_i^k G_i(x^k))\nabla H_i(x^k) \\
            &+ \sum_{i=1}^l \eta^{G,k}_i \nabla G_i(x^k) + \sum_{i \in I_{0+} \cup I_{0-} \cup I_{+-}}(\delta_i^k H_i(x^k))\nabla G_i(x^k)\Big\|_{\infty} \leq \eps_k. 
        \end{split}
    \end{align}

    Define $I := I_{+-} \cup I_{+0} \cup I_{0+} \cup I_{0-}$.
    To prove that the sequence $\{(\eta^{G,k},\eta^{H,k}, \delta_I^k)\}$ is bounded, assume $\|(\eta^{G,k},\eta^{H,k}, \delta_I^k)\| \to \infty$ and w.l.o.g that the complete normalized sequences converges:
    \begin{equation*} 
        \frac{(\eta^{G,k},\eta^{H,k}, \delta_I^k)}{\|(\eta^{G,k},\eta^{H,k}, \delta_I^k)\|} \to (\bar \eta^{G}, \bar \eta^{H}, \bar \delta_I) \neq 0 \quad \text{for } k \to \infty.
    \end{equation*}
    Then, dividing  \eqref{eq:InexactScholtes2}  by $\|(\eta^{G,k},\eta^{H,k}, \delta_I^k)\| $ and passing to the limit we arrive at
   \begin{equation}\label{eq:PosLinDep}
        0 = \sum_{i=1}^l \bar \eta^{H}_i \nabla H_i(x^*)  + \sum_{i=1}^l \bar \eta^{G}_i \nabla G_i(x^*).
   \end{equation}
   Here, we used $G_i(x^k) \to 0$ for $i \in I_{+0}$,   $H_i(x^k) \to 0$ for $i \in I_{0+}\cup I_{0-}$ and $\delta_i^k\to 0$ for $i\in I_{+-}$.
   
   By definition $\supp(\bar \eta^G) \subseteq I_{00} \cup I_{+0}$ and $\supp(\bar \eta^H) \subseteq I_{00} \cup I_{0+} \cup I_{0-}$ because  $\nu^k_i \to 0$ for $i \in I_{+}$.
   To be able to apply MPVC-MFCQ, it remains to verify $\bar \eta^G_i \geq 0$ for all $i \in I_{+0} \cup I_{00}$ and $\bar \eta^H_i \geq 0$ for all $i \in I_{0-}$.
   For all $k$ sufficiently large, we have the following implications
   \begin{eqnarray*}
       i \in I_{0-} & \Longrightarrow & \eta^{H,k}_i = \nu^k_i + \delta^k_i (-G_i(x^k)) \geq -\eps_k (1 - G_i(x^k)) \to 0  \quad  \Longrightarrow \quad \bar \eta^H_i \geq 0, \\
       i \in I_{+0} & \Longrightarrow & \eta^{G,k}_i = \delta^k_i H_i(x^k) \geq -\eps_k H_i(x^k) \to 0 \quad  \Longrightarrow \quad \bar \eta^G_i \geq 0.
   \end{eqnarray*}
   If for some $i \in I_{00}$ we had $\bar \eta^G_i < 0$, then due to $\eta^{G,k}_i = \delta^k_i H_i(x^k)$ we could infer $\delta^k_i \to + \infty$ due to $H_i(x^k) \to 0$. 
   Furthermore, $H_i(x^k) < 0$ has to hold for all $k$ large and thus $\eps_k > 0$ .
   Then  \eqref{eq:InexactSch1} yields $H_i(x^k) \in [- \eps_k,0)$ , which results in a contradiction to  \eqref{eq:deltaUnboundedScholtes}, because it implies
   \begin{equation*}
       \tfrac{1}{\eps_k} |G_i(x^k) H_i(x^k)| \leq |G_i(x^k)| \to 0.
   \end{equation*}

   These properties of $\bar \eta^G, \bar \eta^H$ together with \eqref{eq:PosLinDep}, and the assumption that MPVC-MFCQ holds in $x^*$, yields $\bar \eta^G = \bar \eta^H = 0$ and thus $\bar \delta_I \neq 0$.
   We know $\delta^k_i \to 0$ for all $i \in I_{+-}$ and thus $\bar \delta_{I_{+-}} = 0$.
   However, $\bar \delta_i H_i(x^*) = \bar \eta^G_i = 0$ for all $i \in I_{+0}$ also implies $\bar \delta_{I_{+0}} = 0$.
   Consequently, there has to be an $i \in I_{0+} \cup I_{0-}$ with $\bar \delta_i \neq 0$ and thus $\bar \delta_i > 0$.
   For $i \in I_{0-}$ this together with $\bar \eta^H_i = 0$ yields
   \begin{equation*}
        \lim_{k \to \infty} \frac{\nu_i^k}{\|(\eta^{G,k},\eta^{H,k}, \delta_I^k)\|}  = \bar \delta_i G_i(x^*) < 0,
   \end{equation*}
   a contradiction to \eqref{eq:InexactSch1}.
   The only remaining possibility is thus $\bar \delta_i > 0$ for some $i \in I_{0+}$.
   Again using $\bar \eta^H_i = 0$ we then know
   \begin{equation*}
        \lim_{k \to \infty} \frac{\nu_i^k}{\|(\eta^{G,k},\eta^{H,k}, \delta_I^k)\|}  = \bar \delta_i G_i(x^*) > 0.
   \end{equation*}
   According to our construction, this implies $\nu_i^k \to \infty$ and thus, by \eqref{eq:InexactSch1}, we have $|H_i(x^k)| \leq \frac{\eps_k}{\nu_i^k} \to 0$.
   Since both $\nu_i^k \to \infty$ and $\bar \delta_i > 0$ (hence $\delta_i^k\to \infty$), we know $\eps_k > 0$ for all $k$ large and 
   \begin{equation*}
       \tfrac{1}{\eps_k} |G_i(x^k)H_i(x^k)| \leq |G_i(x^k)| \tfrac{1}{\nu^k_i} \to 0,
    \end{equation*}
    a contradiction to \eqref{eq:deltaUnboundedScholtes}.
    This shows $\bar \delta_I = 0$.

   Consequently, the assumption that $(\eta^{G,k},\eta^{H,k},\delta^k_I)$ is unbounded has to    false.
   We can thus assume without loss of generality that the whole sequence $\{(\eta^{G,k},\eta^{H,k},\delta^k_I)\}$ converges to some limit $(\eta^{G,*}, \eta^{H,*}, \delta_I^*)$.
   Thanks to \eqref{eq:InexactScholtes2}, this limit then satisfies
   \begin{align}
        \begin{split}\label{eq:InexactScholtes3}
            \nabla f(x^*) & - \sum_{i=1}^l \eta^{H,*}_i \nabla H_i(x^*) + \sum_{i=1}^l \eta^{G,*}_i \nabla G_i(x^*)  = 0,
        \end{split}
    \end{align}
    where we again used $H_i(x^k) \to 0$ for $i \in I_{0+}\cup I_{0-}$, $G_i(x^k) \to 0$ for $i \in I_{+0}$ as well as $\delta_i^k\to 0$ for $i\in I_{+-}$.
    By what was proven above, the limit also satisfies
    \begin{equation*}
        \supp( \eta^{G,*}) \subseteq I_{00} \cup I_{+0}
        \quad \text{and} \quad
        \supp( \eta^{H,*}) \subseteq I_{00} \cup I_{0+} \cup I_{0-}
    \end{equation*}
    as well as $\eta^{H,*}_i \geq 0$ for all $i \in I_{0-}$ and $\eta^{G,*}_i \geq 0$ for all $i \in I_{+0} \cup I_{00}$.
    This ensures weak stationarity of $x^*$.

    In order to prove that $x^*$ is in fact T-stationary, assume that there were $i \in I_{00}$ with $\eta^{G,*}_i\eta^{H,*}_i > 0$.
    In case $\nu^k_i \to 0$, this implies
    \begin{equation*}
        \eta^{G,*}_i\eta^{H,*}_i = \lim_{k \to \infty} - (\delta^k_i)^2 G_i(x^k)H_i(x^k) > 0
    \end{equation*}
    and thus $G_i(x^k)H_i(x^k) < 0$ for all $k$ sufficiently large.
    Due to $t_k > 0$ this implies
    \begin{equation*}
        |\delta^k_i| t_k \leq |\delta^k_i| |G_i(x^k)H_i(x^k) - t_k| \leq \eps_k.
    \end{equation*}
    But since $|\delta^k_i| \to \infty$, that is a contradiction to $\eps_k = O(t_k)$.

    In case $\nu_i^k \not\to 0$, we can use that the assumption $\eta^{G,*}_i = \lim_{k \to \infty} \delta^k_i H_i(x^k) \neq 0$ implies $H_i(x^k) \neq 0$ for all $k$ large to infer $\eps_k > 0$ due to \eqref{eq:InexactSch1}.
    Then \eqref{eq:InexactSch1} guarantees $|H_i(x^k)| \leq \frac{\eps_k}{\nu_i^k}$ and thus again
    \begin{equation*}
        \tfrac{1}{\eps_k} |G_i(x^k)H_i(x^k)| \leq |G_i(x^k)| \tfrac{1}{\nu^k_i} \to 0,
    \end{equation*}
    in contradiction to \eqref{eq:deltaUnboundedScholtes}.
    Consequently our assumption was wrong and $\eta^{G,*}_i\eta^{H,*}_i \leq 0$ for all $i \in I_{00}$.
    Hence, $x^*$ is even T-stationary.
\end{proof}

Since we can use $\eps_k = 0$  in the previous theorem, the following exact convergence result is an immediate corollary. Note that this constitutes an improvement over a result by Achtziger et al. \cite[Theorem 3.3(a)]{AHK12}.

\begin{corollary}\label{cor:globalConvergenceExactMFCQ}
    Let $\{t_k\} \downarrow 0$ and $\{x^k\}$ a sequence of KKT points of $R^S(t_k)$ with $x^k \to x^*$.
    If MPVC-MFCQ holds in $x^*$, then $x^*$ is a T-stationary point of the MPVC \eqref{eq:MPVC}.
\end{corollary}

As the regularization method is based on the assumption that local minima of the regularized problems are KKT points, we conclude our discussion by verifying that MPVC-MFCQ locally ensures MFCQ for the regularized problems.
A similar result under the stronger assumption of MPVC-LICQ can be found in \cite[Theorem 5.3]{IzS09}.

\begin{theorem}\label{thm:globalStandardCQMFCQ}
    Let $x^*$ be feasible for the MPVC \eqref{eq:MPVC} such that MPVC-MFCQ holds in $x^*$.
    Then there exists a neighborhood $N(x^*)$ of $x^*$ and a $\bar t > 0$ such that for all $t \in (0,\bar t]$ and all $x \in N(x^*) \cap X^S(t)$ standard MFCQ for $R^S(t)$ is satisfied  in $x$.
\end{theorem}

\begin{proof}
    By continuity, there exists a neighborhood $N(x^*)$ of $x^*$ and a $\bar t > 0$ such that for all $t \in (0,\bar t]$ and all $x \in N(x^*) \cap X^S(t)$ we have the inclusions
    \begin{equation}\label{eq:globalMFCQinclusions}
           I_H(x) \subseteq I_{0}, \quad
           I_{GH}(x,t) \subseteq I_{00} \cup I_{0+} \cup I_{+0},
           \quad \text{and} \quad
            I_{H}(x) \cap I_{GH}(x,t) = \emptyset.
    \end{equation}
    Here, the last conditions follows directly from the definition of the regularized problem.
    Since MPVC-MFCQ holds in $x^*$, we know that the gradients
    $$
        (- \nabla H_i(x)) \; (i \in I_{0-}), \quad
        \nabla G_{i} (x) \; (i \in I_{00} \cup I_{+0}),
        \quad \text{and} \quad
        \nabla H_i(x) \; (i \in I_{00} \cup I_{0+})
    $$  
    are positively linearly independent in $x^*$.
    In view of \cite[Proposition 2.2]{QiW00}, this implies that these gradients remain positively linearly independent for all $x \in N(x^*) \cap X^S(t)$, if $N(x^*)$ is chosen sufficiently small.
    We trivially have the inclusions
    \begin{align*}
        & (I_{GH}(x,t) \cap I_{00}) \cup (I_{GH}(x,t) \cap I_{+0}) \subseteq I_{00} \cup I_{+0}, \\
        & (I_{GH}(x,t) \cap I_{00}) \cup (I_{GH}(x,t) \cap I_{0+}) \subseteq I_{00} \cup I_{0+}, \\
        & I_H(x) \cap (I_{00} \cup I_{0+}) \subseteq I_{00} \cup I_{0+},\\
        & I_H(x) \cap I_{0-} \subseteq I_{0-}
    \end{align*}
    for all $x \in N(x^*) \cap X^S(t)$.
    Exploiting  $G_i(x) > 0, H_i(x) \approx 0$ for $i \in I_{0+}$ as well as $G_i(x) \approx 0, H_i(x) > 0$ for $i \in I_{+0}$ and shrinking $N(x^*)$ further, if necessary, we can ensure that for all $x \in X^S(t) \cap N(x^*)$ the gradients
    \begin{equation*}
        \begin{array}{rrl}
            & -\nabla H_i(x) & (i \in I_{H}(x)\cap I_{0-}) \\
            & \nabla G_i(x) & (i \in I_{GH}(x,t) \cap I_{00})  \\
            & H_i(x)\nabla G_i(x) + G_i(x)\nabla H_i(x) & (i \in I_{GH}(x,t) \cap I_{+0})  \\
            \text{and }&\nabla H_i(x) & (i \in I_{GH}(x,t) \cap I_{00}) \\
            & H_i(x)\nabla G_i(x) + G_i(x)\nabla H_i(x) &(i \in I_{GH}(x,t) \cap I_{0+}) \\
            & \nabla H_i(x) & (i \in I_{H}(x)\cap (I_{00} \cup I_{0+})) \\
        \end{array}
    \end{equation*}
    are positively linearly independent.
    But then thanks to the inclusions \eqref{eq:globalMFCQinclusions} for all $x \in N(x^*) \cap X^S(t)$ the gradients
    \begin{equation}\label{eq:PLIndep}
        \begin{array}{rrl}
            & -\nabla H_i(x) & (i \in I_{H}(x)) \\
            & H_i(x)\nabla G_i(x) + G_i(x)\nabla H_i(x) & (i \in I_{GH}(x,t))  
        \end{array}
    \end{equation}
    are positively linearly independent, which is  MFCQ for the regularized problem $R^S(t)$ in $x$.
\end{proof}
\subsection{The Local Regularization Scheme}\label{ssec:localRegularization}

While the previously considered global regularization relaxed the vanishing constraint globally, the next regularization method, which was introduced by Steffensen and Ulbrich \cite{StU10} for MPCCs and used for MPVCs by Hoheisel et al. \cite{HKS12local}, relaxes the vanishing constraint only locally around the origin.
For a given regularization parameter $t> 0$, the regularized problem reads

\begin{minipage}[c]{0.5\textwidth}
    \begin{align} 
        \min \ & f(x) \nonumber\\ 
        \st \ & g_i(x) \leq 0 \  (i = 1,\ldots,m), \nonumber  \\ 
         & h_j(x) = 0 \ (j = 1,\ldots,p), \tag{$R^{SU}(t)$} \nonumber \\ 
         & H_i(x) \geq 0 \ (i = 1,\ldots,l),  \nonumber \\ 
         & \Phi^{SU}_i(x;t) \leq 0 \  (i = 1,\ldots,l),  \nonumber 
    \end{align}
\end{minipage}
\qquad
\begin{minipage}[c]{0.45\textwidth}
    \begin{tikzpicture}[scale = 1.25]
        \fill[blue!30] (-1.5,0) -- (0,0) -- (0,2) -- (-1.5,2) -- cycle;
        \fill[blue!30]  (1,0) .. controls (0,0.0) and (0.0,0) .. (0,1) -- (0,0) -- cycle;
        \draw[->] (-1.75,0) -- (1,0) node[below]{$t$} -- (2.5,0) node[below]{$G_i(x)$};
        \draw[->] (0,-0.5)  -- (0,1) node[left]{$t$} -- (0,2.5) node[left]{$H_i(x)$};
        \draw[line width = 1.5pt, blue] (-1.5,0) -- (0,0) -- (2,0);
        \draw[line width = 1.5pt, blue] (1,0) .. controls (0,0.0) and (0.0,0) .. (0,1) -- (0,2);
        \draw[->] (0.5,1.5) node[right]{$\Phi^{SU}_i(x;t)=0$} -- (0,0.75);
    \end{tikzpicture}
\end{minipage}

\noindent
with $\Phi_i^{SU}: \R^n \rightarrow \R$ defined as
$$
    \Phi_i^{SU} (x;t) = G_i(x) + H_i(x) - \varphi(G_i(x) - H_i(x);t)
$$
where $\varphi:\R \to \R$ is defined as
\begin{equation*}
    \varphi(a;t) :=
    \begin{cases}
        |a|   & \text{if } |a| \geq t,\\
        t \theta \left( \frac{a}{t} \right) &  \text{if } |a| < t
    \end{cases}
\end{equation*}
and $\theta:[-1,1] \rightarrow \R$ is a function satisfying the following conditions:
\begin{itemize}
    \item[(a)] $\theta$ is twice continuously differentiable on [-1,1];
    \item[(b)] $\theta(-1) = \theta(1) = 1$;
    \item[(c)] $\theta' (-1) = -1$ and $\theta' (1) = 1$;
    \item[(d)] $\theta'' (-1) = \theta''(1) = 0$;
    \item[(e)] $\theta''(x) > 0$ for all $x \in (-1,1)$.
\end{itemize}
We denote the feasible set of $R^{SU}(t)$ by  $X^{SU}(t)$.  

The following lemma collects some important properties of the function $\Phi_i^{SU}$.

\begin{lemma}[{{\cite[Lemma 4.5-4.6]{HKS12local}}}] \label{lem:localPropertiesPhi}
    For $t>0$  and $x\in \R^n$ we have for all $i=1,\ldots,l$
    \begin{equation*}
        \Phi_i^{SU}(x;t) 
        \begin{cases}
            < 0 & \text{if } \min\{G_i(x),H_i(x)\} < 0, \\
            < 0 & \text{if } \min\{G_i(x),H_i(x)\} = 0 \text{ and } |G_i(x) - H_i(x)| < t, \\
            = 0 & \text{if } \min\{G_i(x),H_i(x)\} = 0 \text{ and } |G_i(x) - H_i(x)| \geq t, \\
            > 0 & \text{if } \min\{G_i(x),H_i(x)\} > 0 \text{ and } |G_i(x) - H_i(x)| \geq t, \\
        \end{cases}
    \end{equation*} 
    and
    $$
        \nabla\Phi_i^{SU}(x; t) = \alpha_i \nabla G_i(x) + \beta_i \nabla H_i(x)
    $$
    with
    \begin{equation*}
       (\alpha_i, \beta_i) = 
       \begin{cases}
        \binom{2}{0} & \text{if } G_i(x) - H_i(x) \leq -t, \\
        \binom{0}{2} & \text{if } G_i(x) - H_i(x)  \geq t, \\
        \left(\begin{smallmatrix}1-\theta'\left(\frac{G_i(x)-H_i(x)}{t}\right)\\1+ \theta'\left(\frac{G_i(x)-H_i(x)}{t}\right)\end{smallmatrix}\right), & \text{if } |G_i(x) - H_i(x)| < t.
       \end{cases}
    \end{equation*}
\end{lemma}

The following exact convergence result has been established by Hoheisel et al. \cite{HKS12local}.

\begin{theorem}[{{\cite[Theorem 4.12]{HKS12local}}}] \label{thm:localConvergenceExactCPLD}
    Let $\{t_k\} \downarrow 0$ and let  $\{x^k\}$ be a sequence of KKT points of $R^{SU}(t_k)$ with $x^k\to x^*$ such that MPVC-CPLD holds at $x^*$.
    Then $x^*$ is a T-stationary point of the MPVC \eqref{eq:MPVC}.
\end{theorem}

To facilitate the subsequent analysis of the inexact case, we need some new index sets.
For $t>0$ and $x \in X^{SU}(t)$ we define
\begin{eqnarray*}
   I_H(x)&:=&\{i\mid H_i(x)=0\},\\
   I_{\Phi}(x,t)&:=&\{i\mid \Phi_i^{SU}(x; t) = 0\}.
\end{eqnarray*}

Then analogously to the MPCC case, we see that in the case of $\eps$-stationary points of the regularized problems, the theoretical convergence properties of this scheme deteriorate.

\begin{theorem}\label{thm:localConvergenceEpsMFCQ}
    Let $\{t_k\}, \{\eps_k\}  \downarrow 0$, and $\{x^{k}\}$ be a sequence of $\eps_k$-stationary points of $R^{SU}(t_{k})$ with $x^k \rightarrow x^*$ such MPVC-MFCQ holds at $x^{*}$.
    Then $x^{*}$ is a weakly stationary point of \eqref{eq:MPVC}.
\end{theorem}

\begin{proof}
    Since $x^{k}$ are $\eps_k$-stationary points of $R^{SU}(t_{k})$, there exist multipliers $\{(\nu^k, \delta^k)\}$ such that for all $k\in \N$
    \begin{eqnarray}
        && \left\| \nabla f(x^k) - \sum^{l}_{i=1} \nu_i^k \nabla H_i(x^k) + \sum^{l}_{i=1} \delta_i^k \nabla \Phi^{SU}_{i} (x^k;t_k)  \right\|_{\infty} \, \leq \varepsilon_k, \label{eq:SU1} \\
        && H_i(x^k) \geq -\eps_k, \quad \nu_i^k \geq -\eps_k, \quad |\nu_i^kH_i(x^k)| \leq \eps_k \quad (i = 1,\dots,l), \label{eq:SU2} \\
        && \Phi^{SU}_i (x; t_k) \leq \eps_k, \quad \delta_i^{k} \geq -\eps_k, \quad |\delta_i^{k} \Phi^{SU}_i (x; t_k)| \leq \eps_k \quad (i = 1,\dots, l). \label{eq:SU3} 
    \end{eqnarray}
    In particular,  $x^*$ is  feasible for  \eqref{eq:MPVC}.
    We now observe that, for all $k\in \N$ and $i=1\dots,l$, we have 
    \[
        \nabla\Phi_i^{SU}(x^k);t_k=\alpha_i^k \nabla G_i(x^k)+\beta_i^kH_i(x^k)
    \]
    with $(\alpha_i^k, \beta_i^k)$ given by Lemma \ref{lem:localPropertiesPhi}.
    Hence, defining the multipliers
    $$
         \eta_i^{H,k} := \nu_i^{k} - \delta_i^{k}\beta_i^k
         \quad \text{and} \quad
         \eta_i^{G,k}  := \delta_i^k \alpha_i^k
    $$
    for all $i = 1,\dots,l$, we can rewrite \eqref{eq:SU1} as
    \begin{align}\label{eq:SU1a}
        \left\| \nabla f(x^k) - \sum^{l}_{i=1} \eta_i^{H,k}  \nabla H_i(x^k) + \sum^{l}_{i=1} \eta_i^{G,k}  \nabla G_i(x^k)  \right\|_{\infty} \leq \varepsilon_k.
    \end{align}
    We claim that the sequence $\{(\eta^{G,k}, \eta^{H,k})\}$ is bounded.
    Otherwise, we may assume w.l.o.g. that the whole normalized sequence converges:
    \[
        \frac{(\eta^{G,k}, \eta^{H,k})}{\|(\eta^{G,k}, \eta^{H,k})\|} \longrightarrow (\bar \eta^{G}, \bar \eta^{H}) \neq 0.
    \]
    Dividing (\ref{eq:SU1a}) by $\|(\eta^{G,k}, \eta^{H,k})\|$ and passing to the limit then yields
    \begin{align} \label{eq:SU4}
        - \sum^{l}_{i=1} \bar \eta^{H}_i \nabla H_i(x^k) + \sum^{l}_{i=1} \bar \eta^{G}_i \nabla G_i(x^k) = 0.
    \end{align}

    We now  show that  $\supp(\bar \eta^{G}) \subseteq I_{00}\cup I_{+0}$.
    For all $i\in I_{0+}$, we have $H_i(x^*) = 0$ and $G_i(x^*) > 0$.
    This implies $G_i(x^k) - H_i(x^k) \geq t_k$ for $k$ sufficiently large and thus, by Lemma \ref{lem:localPropertiesPhi}, $\alpha_i^k = 0$.
    It then follows that $\eta_i^{G,k} = \alpha_i^k \delta_i^k = 0$ for all $k$ large and therefore $\bar \eta^{G}_i=0$.

    For all  $i \in I_{0-}\cup I_{+-}$,  we  have $H_i(x^{*}) \geq 0$ and $G_i(x^{*}) < 0$.
    This implies $G_i(x^k) - H_i(x^k) \leq -t_k$ for $k$ sufficiently large and by Lemma \ref{lem:localPropertiesPhi}, we then have  $\alpha_i^k = 2$.
    It follows that $\eta_i^{G,k} = \alpha_i^k \delta_i^k = 2\delta_i^k$, and hence $\bar \eta_i^G \neq 0$ is only possible if $|\delta_i^k|  \not \rightarrow 0$.
    By \eqref{eq:SU3}, this would imply $\Phi^{SU}_i (x^k; t_k) \rightarrow 0$ and thus for $k$ sufficiently large
    \[
        2G_i(x^k) = G_i(x^k) + H_i(x^k) - |G_i(x^k) - H_i(x^k)| =G_i(x^k) + H_i(x^k) - \varphi(G_i(x^k) - H_i(x^k); t)\rightarrow 0.
    \]
    But this would contradict $i \in I_{0-}\cup I_{+-}$.
    All in all, we have shown  $\supp(\bar \eta^{G}) \subseteq I_{00}\cup I_{+0}$.

    The next step is to show $\supp(\bar \eta^{H}) \subseteq I_{0}$.
    For all $i \in I_+ = I_{+-} \cup I_{+0}$ we have $H_i(x^{*}) >0$ and $G_i(x^{*}) \leq 0$.
    Then for all $k$ sufficiently large $G_i(x^k) - H_i(x^k) \leq -t_k$  follows and Lemma \ref{lem:localPropertiesPhi} yields $\beta_i^k = 0$.
    This implies $\eta^{H,k}_i = \nu_i^k$ for all $k$ sufficiently large.
    Thus, $\bar \eta^H_i \neq 0$ is only possible, if $|\nu^k_i| \neq 0$.
    But then $H_i(x^*) = 0$ has to hold due to \eqref{eq:SU2} , which contradicts $i\in I_{+}$.
    
    Using the previous observations,  \eqref{eq:SU4} reduces to 
    \begin{align*}
        - \sum_{i \in I_0} \bar \eta^{H}_i \nabla H_i(x^k) + \sum_{i \in I_{00} \cup I_{0+}} \bar \eta^{G}_i \nabla G_i(x^k) = 0.
        \end{align*}
    We now observe that, for $i\in I_{0-}$ and  $k$ sufficiently large, we have $G_i(x^k)-H_i(x^k)<-t_k$  and hence $\beta^k_i = 0$ and $ \eta_i^{H,k}=\nu_i^k\geq -\eps_k$.
    This guarantees $\bar \eta^{H}_i \geq 0$ for $i\in I_{0-}$.
    Since $\eta_i^{G,k} = \alpha_i^k\delta_i^k$ with $\alpha_i^k\in [0,2]$ and $\delta_i^k\geq -\eps_k$ for all $k\in \N$ and all $i=1,\dots,l$, we also know  $\bar \eta^{G}_i \geq 0$ for all $i\in I_{00}\cup I_{+0})$.  
    But then \eqref{eq:SU4} together with $(\bar \eta^{G}, \bar \eta^{H}) \neq 0$ contradicts MPVC-MFCQ at $x^*$ .
    
    Therefore, the sequence $\{(\eta^{G,k}, \eta^{H,k})\}$ is bounded and, at least on a subsequence,  converges to a limit $(\eta^{G,*}, \eta^{H,*})$.
    Reiterating the previous arguments proves that $(x^*,\eta^{G,*}, \eta^{H,*})$ is a  weakly stationary point.
\end{proof}

Consequently, even under the stronger assumption of MPVC-MFCQ, using only $\eps$-stationary points of the regularized problems, we cannot guarantee T-stationarity of the limit anymore.
This is analogous to the MPCC case discussed in \cite{KaS15}.
In this reference, two MPCC examples are provided to illustrate that in the inexact setting limits, which are only weakly stationary, can actually occur.
Those examples can also be translated into the MPVC setting.

Finally, we want to close our discussion of this local regularization method by again proving that the regularized problems locally inherit a constraint qualification from the MPVC.
However, since the vanishing constraint is relaxed only locally, we cannot expect strong CQs such as LICQ or MFCQ for the regularized problem if $I_{0+} \neq \emptyset$.

\begin{theorem}\label{thm:localStandardCQACQ}
    Let $x^*$ be feasible for the MPVC \eqref{eq:MPVC} such that MPVC-LICQ holds at $x^*$.
    Then there exists $\bar{t} > 0$ and a neighborhood $N(x^*)$ of $x^*$ such that, for all $t \in (0, \bar{t}]$ and all $x \in X^{SU}(t) \cap N(x^*)$, standard ACQ for $R^{SU}(t)$ is satisfied at $x$.
\end{theorem}

\begin{proof}
    By continuity, there exists a neighborhood $N(x^*)$ of $x^*$ and a $\bar t > 0$ such that for all $t \in (0,\bar t]$ and all $x \in N(x^*) \cap X^{SU}(t)$ we have the inclusions
    \begin{equation}\label{eq:localACQinclusions}
           I_H(x) \subseteq I_{0}, \quad
           I_G(x) \subseteq I_{+0} \cup I_{00}, \quad
           I_{\Phi}(x,t) \subseteq I_{00} \cup I_{0+} \cup I_{+0}.
    \end{equation}
    Shrinking $\bar t$ and $N(x^*)$ if necessary, we can also achieve
    \begin{equation}\label{eq:localACQinclusions2}
        I_{0+} \subseteq I_\Phi(x,t) \cap I_{H}(x)
    \end{equation}
    for all $t \in (0,\bar t]$ and all $x \in N(x^*) \cap X^{SU}(t)$.
    
    Now consider an arbitrary $t \in (0,\bar t]$ and  $\hat x \in N(x^*) \cap X^{SU}(t)$.
    We define the auxiliary problem NLP$(\hat{x})$ by
    \begin{eqnarray*}
       \min f(x) & \st & H_i(x) \geq 0 \quad (i \in I_{H}(\hat{x}) \setminus I_{\Phi}(\hat{x},t)), \\
       && H_i(x) = 0 \quad (i \in  I_{\Phi}(\hat{x},t) \cap I_{H}(\hat{x})), \\ 
       && \Phi_i^{SU}(x;t) \leq 0 \quad (i \in  I_{\Phi}(\hat{x},t) \setminus I_{H}(\hat{x})),
    \end{eqnarray*}
    and denote its feasible region by $\hat{X}$.
    Then, clearly, $\hat{x} \in \hat{X}$.
    
    Our next step is to prove that LICQ for NLP$(\hat{x})$ holds in $\hat x$.
    To this end, note that the gradients of the active constraints are 
    \begin{eqnarray*}
        \nabla H_i(\hat x) && (i \in I_H(\hat x)  \subseteq I_0), \\
        \alpha_i \nabla G_i(\hat x) + \beta_i \nabla H_i(\hat x) && (i \in I_\Phi(\hat x,t) \setminus I_H(\hat x) \subseteq I_{00} \cup I_{+0}),
    \end{eqnarray*}
    with $\alpha_i,\beta_i$ given by Lemma \ref{lem:localPropertiesPhi}.
    Here, we used $I_{0+} \subseteq I_\Phi(\hat x,t) \cap I_{H}(\hat x)$ by the choice of $\bar t$ and $N(x^*)$.
    Consequently, if we choose $N(x^*)$ small enough, MPVC-LICQ implies linear independence of the above gradients and thus LICQ.
    
    Since LICQ implies ACQ for NLP$(\hat x)$, we thus also know $\T_{\hat X}(\hat x) = \Li_{\hat X}(\hat x)$, where
    \begin{eqnarray*}
        \Li_{\hat{X}}(\hat{x}) =   \{ d \in \R^n \mid & \nabla H_i(\hat{x})^T d \geq 0 & (i \in I_H(\hat{x}) \setminus I_{\Phi}(\hat{x},t) ),  \\
        &\nabla H_i(\hat{x})^T d = 0 & (i \in I_H(\hat{x}) \cap I_{\Phi}(\hat{x},t) ),\\
        & (\alpha_i \nabla G_i(\hat x) + \beta_i \nabla H_i(\hat x))^T d \leq 0 & (i \in I_{\Phi}(\hat{x},t) \setminus I_H(\hat{x})  )\} .
    \end{eqnarray*}
    
    If we choose $r > 0$ sufficiently small, then $ \hat X \cap B_r(\hat x) \subseteq X^{SU}(t)$, because for all $i \notin I_H(\hat x)$ we have $H_i(\hat x) > 0$, for all $i \notin I_\Phi(\hat x,t)$ we have $\Phi^{SU}_i(\hat x; t) < 0$ and for all $i \in I_\Phi(\hat x,t) \cap I_H(\hat x)$ the constraint $H_i(x) = 0$ implies $\Phi^{SU}_i(x;t) \leq 0$ by Lemma \ref{lem:localPropertiesPhi}.
    This implies $\Li_{\hat X}(\hat x) =\T_{\hat X}(\hat x) \subseteq \T_{X^{SU}(t)} (\hat x)$.

    To complete the proof, it remains to observe that $\Li_{X^{SU}(t)} (\hat x) = \Li_{\hat X}(\hat x)$, because
    \begin{eqnarray*}
        \Li_{X^{SU}(t)} (\hat x) =   \{ d \in \R^n \mid & \nabla H_i(\hat{x})^T d \geq 0 & (i \in I_H(\hat{x}) ),  \\
        & (\alpha_i \nabla G_i(\hat x) + \beta_i \nabla H(\hat x))^T d \leq 0 & (i \in I_{\Phi}(\hat{x},t) )\}
    \end{eqnarray*}
    where for all $i \in I_H(\hat{x}) \cap I_{\Phi}(\hat{x},t)$ we have $|G_i(\hat x)| \geq t$ and thus $(\alpha_i, \beta_i) = (0,2)$ by Lemma \ref{lem:localPropertiesPhi}. 
\end{proof}
\subsection{The L-shaped Regularization Scheme}\label{ssec:LshapedRegularization}

In this section  we study the so-called $L$-shaped regularization introduced by Kanzow and Schwartz in \cite{KaS13} for MPCCs and adapted for MPVCs in \cite{HKS12new}.
For $t>0$, it is  given by

\begin{minipage}[c]{0.5\textwidth}
    \begin{align} 
        \min \ & f(x) \nonumber\\ 
        \st \ & g_i(x) \leq 0 \  (i = 1,\ldots,m), \nonumber  \\ 
         & h_j(x) = 0 \ (j = 1,\ldots,p), \tag{$R^{KS}(t)$} \nonumber \\ 
         & H_i(x) \geq 0 \ (i = 1,\ldots,l),  \nonumber \\ 
         & \Phi^{KS}_i(x;t) \leq 0 \  (i = 1,\ldots,l),  \nonumber 
    \end{align}
\end{minipage}
\qquad
\begin{minipage}[c]{0.45\textwidth}
    \begin{tikzpicture}[scale = 1.25]
        \fill[blue!30] (-1.5,0) -- (0,0) -- (0,2) -- (-1.5,2) -- cycle;
        \fill[blue!30]  (0,0) -- (2,0) -- (2,0.5) -- (0,0.5) -- cycle;
        \draw[->] (-1.75,0) -- (2.5,0) node[below]{$G_i(x)$};
        \draw[->] (0,-0.5)  -- (0,0.5) node[left]{$t$} -- (0,2.5) node[left]{$H_i(x)$};
        \draw[line width = 1.5pt, blue] (-1.5,0) -- (0,0) -- (2,0);
        \draw[line width = 1.5pt, blue] (0,2) -- (0,0.5) -- (2,0.5);
        \draw[->] (0.5,1.5) node[right]{$\Phi^{KS}_i(x;t)=0$} -- (0,0.75);
    \end{tikzpicture}
\end{minipage}

\noindent
with $\Phi_i^{KS}: \R^n \rightarrow \R$ being defined as the once continuously differentiable function
$$
    \Phi_i^{KS} (x;t) = 
    \begin{cases}
        G_i(x)(H_i(x) - t) & \text{if } G_i(x) + H_i(x) \geq t, \\
        -\frac{1}{2} [G_i(x)^2 + (H_i(x)-t)^2] & \text{if } G_i(x) + H_i(x) < t.
    \end{cases}
$$
We denote the feasible set of $R^{KS}(t)$ by  $X^{KS}(t)$.  

Contrary to the two previously considered regularization schemes, this approach guarantees M-stationarity and not just T-stationarity of limits in the exact case.

\begin{theorem}[{{\cite[Theorem~4.1]{HKS12new}}}]
\label{thm:LconvergenceExactCPLD}
     Let $\{t_k\} \downarrow 0$ and let  $\{x^k\}$ be a sequence of KKT points of $R^{KS}(t_k)$with $x^k\to x^*$ such that MPVC-CPLD holds at $x^*$.
    Then $x^*$ is an M-stationary point of the MPVC \eqref{eq:MPVC}.
\end{theorem}

It was also proven in \cite{HKS12new} that the regularized problems satisfy standard GCQ locally under suitable assumptions.
Since the regularized problems retain the kink of the feasible set in $(G_i(x),H_i(x)) = (0,t)$, one cannot hope for a stronger CQ to be satisfied in all of $ X^{KS}(t)$.

\begin{theorem}[{{\cite[Theorem~4.7, 4.9]{HKS12new}}}] \label{thm:LstandardCQGCQ}
    Let $x^*$ be feasible for MPVC \eqref{eq:MPVC} and MPVC-LICQ hold in $x^*$.
    Then there exists $\bar{t} > 0$ and a neighborhood $N(x^*)$ of $x^*$ such that, for all $t \in (0, \bar{t}]$ and all $x \in X^{KS}(t) \cap N(x^*)$, standard GCQ for $R^{KS}(t)$ is satisfied at $x$.
    If additionally $(G_i(x),H_i(x)) \neq (0,t)$ for all $i=1,\ldots,l$, then standard LICQ holds at $x$.
\end{theorem}

However, similarly to the local regularization, the favourable convergence properties are lost, if one computes only $\eps$-stationary points of the regularized problems.

\begin{theorem}\label{thm:LconvergenceEpsMFCQ}
    Let $\{t_k\}, \{\eps_k\}  \downarrow 0$, and let $\{x^{k}\}$ be a sequence of $\eps_k$-stationary points of $R^{KS}(t_{k})$ with $x^k \rightarrow x^*$ such MPVC-MFCQ holds at $x^{*}$.
    Assume furthermore that 
    \begin{equation}\label{eq:LepsLimsup}
        \liminf_{k\to\infty} \frac{\eps_k}{G_i(x^k)}\leq 0\quad (i\in I_{00})
    \end{equation}
    whenever this is well-defined.
    Then $x^{*}$ is a weakly stationary point of \eqref{eq:MPVC}.
\end{theorem}

\begin{proof}
    Since $x^k$ are $\eps_k$-stationary points of $R^{KS}(t_k)$, there exist multipliers $(\nu^k,\delta^k)$ such that
    \begin{eqnarray}
        && \left\|\nabla f(x^k)- \sum_{i=1}^k\nu_i^k \nabla H_i(x^k) + \sum_{i=1}^l \delta_i^k \nabla \Phi_i^{KS}(x^k; t_k)\right\|_\infty \leq \eps_k, \label{eq:L1} \\
        && \nu_i^k \geq -\eps_k, \quad H_i(x^k) \geq -\eps_k, \quad |\nu^k_i H_i(x^k)| \leq \eps_k \quad (i = 1,\ldots, l), \label{eq:Mult1}\\
        && \delta_i^k \geq -\eps_k, \quad \Phi_i^{KS}(x^k; t_k) \leq \eps_k, \quad |\delta_i^k \Phi_i^{KS}(x^k; t_k)| \leq \eps_k \quad (i = 1,\ldots, l). \label{eq:Mult2}
    \end{eqnarray}
    Here, the gradients of $\Phi_i^{KS}$ are given by 
    \[
        \nabla \Phi_i^{KS}(x^k ; t_k) =
        \begin{cases}
            (H_i(x^k) - t_k) \nabla G_i(x^k) + G_i(x^k)\nabla H_i(x^k) & \text{if } G_i(x^k) + H_i(x^k) \geq t_k, \\
            -G_i(x^k) \nabla G_i(x^k) - (H_i(x^k) - t_k)\nabla H_i(x^k) & \text{if } G_i(x^k) + H_i(x^k) < t_k.
        \end{cases}
    \]
    
    We now define multipliers $\eta^{G,k}, \eta^{H,k} \in \R^l$ component-wise  by
    \begin{align*}
        &\eta^{G,k}_i :=
        \begin{cases}
            \delta_i^k (H_i(x^k) - t_k) & \text{if } G_i(x^k) + H_i(x^k) \geq t_k, \\
            - \delta_i^k G_i(x^k) & \text{if } G_i(x^k) + H_i(x^k) < t_k,
        \end{cases} \\
        &\eta_i^{H,k} :=
        \begin{cases}
            \nu_i^k - \delta_i^k G_i(x^k) & \text{if } G_i(x^k) + H_i(x^k) \geq t_k, \\
            \nu_i^k + \delta_i^k (H_i(x^k) - t_k)& \text{if } G_i(x^k) + H_i(x^k) < t_k.
        \end{cases}
    \end{align*}
    Using these new multipliers, we can re-write \eqref{eq:L1} as
    \begin{align}\label{eq:L2}
        \left\|\nabla f(x^k) - \sum_{i=1}^l \eta^{H,k}_i \nabla H_i(x^k) + \sum_{i=1}^l \eta^{G,k}_i \nabla G_i(x^k) \right\|_{\infty} \leq \eps_k. 
    \end{align}
    
    We claim that the sequence $\{(\eta^{G,k}, \eta^{H,k})\}$ is bounded.
    If the sequence were unbounded, then w.l.o.g. we have 
    \begin{equation}\label{eq:BddConv}
        \frac{(\eta^{G,k},\eta^{H,k})}{\|(\eta^{G,k},\eta^{H,k})\|} \to (\bar{\eta}^G, \bar{\eta}^H) \neq 0.
    \end{equation}
    Therefore, dividing \eqref{eq:L2} by $\|(\eta^{G,k},\eta^{H,k})\|$ and passing to the limit we obtain
    \begin{align}\label{eq:L3}
        0 = \sum_{i=1}^l \bar{\eta}_i^G \nabla G_i(x^{*}) - \sum_{i=1}^l \bar{\eta}^{H}_i\nabla H_i(x^{*}).    
    \end{align}
    We now determine the support of the multipliers:
    For all  $i \in I_{0+}$ we have $G_i(x^k) + H_i(x^k) \geq t_k$ for all $k$ sufficiently large.
    By the definition of $\eta_i^{G,k}$ and \eqref{eq:Mult2} we hence have
    \[
        G_i(x^k)\eta_i^{G,k} = G_i(x^k) \delta_i^k (H_i(x^k) - t_k) = \delta_i^k \Phi_i^{KS}(x^k;t_k) \to  0, 
    \]
    Since $G_i(x^*)>0$, this implies  $\eta_i^{G,k}\to 0$ and hence $\bar{\eta}^G_i = 0$.
    For all $i \in I_{0-}\cup I_{+-}$, we have $\lim_{k\to \infty} \Phi_i^{KS}(x^k;t_k) \neq 0$, and hence $\delta_i^k\to 0$ due to  \eqref{eq:Mult2}.
    Therefore, $\eta_i^{G,k}\to 0$ and hence $\bar{\eta}^G_i = 0$.
    All in all, we have proven $\supp(\bar{\eta}^{G}) \subseteq I_{00} \cup I_{+0}$. 
    
    For all $i \in I_{+}$, we know $G_i(x^k) + H_i(x^k) > t_k$ for  all  $k$ sufficiently large and thus $\eta_i^{H,k} = \nu_i^k - \delta_i^k G_i(x^k)$.
    Since we have $\nu_i^k \to 0$ by \eqref{eq:Mult1}, we can use the definition of $\eta_i^{H,k}$ to conclude
    \[ 
        \eta_i^{H,k}(H_i(x^k) - t_k) = (\nu_i^k - \delta_i^k G_i(x^k))(H_i(x^k) - t_k) = \nu_i^k(H_i(x^k) - t_k)- \delta_i^k \Phi_i^{KS}(x^k; t_k)\to 0
    \] 
    and thus due to $H_i(x^*) >0 $ obtain $\eta_i^{H,k}\to 0$.
    This shows $\supp(\bar{\eta}^H) \subseteq I_{0-} \cup I_{0+} \cup I_{00}$.
    Hence, \eqref{eq:L3} reduces to
    \begin{align*}
        0 = \sum_{i\in I_{00} \cup I_{+0}} \bar{\eta}_i^G \nabla G_i(x^*) - \sum_{i \in I_{00} \cup I_{0-} \cup I_{0+}} \bar{\eta}^{H}_i \nabla H_i(x^*).  
    \end{align*}
    
    The next step is to determine the signs of the multipliers $\bar \eta_i^G\;(i\in I_{00}\cup I_{+0})$ and $\bar \eta_i^H\;(i\in I_{0-})$.
    To this end, first consider $i\in I_{0-}$ and assume that $\bar \eta_i^{H}<0$.
    Then for $k$ sufficiently large, we know  $G_i(x^k) + H_i(x^k) < t_k$ and thus,  using \eqref{eq:Mult1},
    \[
        \eta_i^{H,k} = \delta_i^k(H_i(x^k)-t_k) + \nu_i^k \leq -c < 0
    \] 
    for all $k$ large.
    Due to $\nu_i^k \geq -\eps_k \to 0$ and $\delta_i^k \geq - \eps_k \to 0$, it follows that $ \delta_i^k \to +\infty$.
    From \eqref{eq:Mult2} together with $i\in I_{0-}$, we hence obtain for $k$ sufficiently large
    \begin{equation*}
        \eps_k \geq  |\delta_i^k \Phi_i^{KS}(x^k;t_k)|
        = \delta_i^k \tfrac{1}{2}[G_i(x^k)^2 + (H_i(x^k)-t_k)^2]
        \geq \tfrac{1}{2}\delta_i^k G_i(x^k)^2
        \to  +\infty,
    \end{equation*}
    which is a contradiction to $\eps_k \to 0$.
    This shows $\bar \eta_i^{H}\geq 0$ for all $i\in I_{0-}$.
    
    Now consider $i\in I_{+0}$.
    Then  $G_i(x^k) + H_i(x^k) \geq t_k$ and thus $\eta_i^{G,k}=\delta_i^k(H_i(x^k)-t_k)$ for all $k$ large.
    As $\delta_i^k\geq -\eps_k$ and $H_i(x^k)-t_k\to H_i(\bar x)>0$, this implies $\bar \eta_i^G \geq 0$ for all $i\in I_{+0}$.
    
    
    Finally, consider $i \in I_{00}$ and assume that $\bar \eta_i^G<0$.
    If $G_i(x^k)+H_i(x^k)\geq t_k$ for almost all $k$, then for some $c>0$ we have
    $$
       \eta_i^{G,k}=\delta_i^k(H_i(x^k)-t_k) \leq -c < 0, \quad
       \delta^k_i \to + \infty
       \quad \text{and} \quad
       - \eta_i^{G,k} G_i(x^k) \leq |\delta_i^k \Phi_i^{KS}(x^k; t_k)| \leq \eps_k
    $$
    for all $k$ large.
    If instead $G_i(x^k)+H_i(x^k) <  t_k$ for almost all $k$, then analogously to the previous discussion we obtain  for all $k$ large
    $$
       \eta_i^{G,k}= -\delta_i^k G_i(x^k) \leq -c < 0, \quad
       \delta^k_i \to + \infty
       \quad \text{and} \quad
       - \eta_i^{G,k} G_i(x^k) \leq 2|\delta_i^k \Phi_i^{KS}(x^k; t_k)| \leq 2 \eps_k.
    $$
    Since in both cases $G_i(x^k)$ has to be positive for all $k$ large due to $\eta_i^{G,k} < 0$, it follows that
    $$
        \frac{2\eps_k}{G_i(x^k)} \geq -\eta_i^{G,k} \geq c >0
    $$
    for all $k$ large, which contradicts the assumption \eqref{eq:LepsLimsup}.
    
    All in all, we have shown $\bar \eta_i^G\geq 0 \; (i\in I_{00}\cup I_{+0})$ and $\bar \eta_i^H \geq 0 \; (i\in I_{0-})$.
    But this together with \eqref{eq:L3} and \eqref{eq:BddConv} contradicts the assumption that MPVC-MFCQ holds in $x^*$.
    
    Therefore the sequence $\{( \eta^{G,k}, \eta^{H,k})\}$ is bounded and w.l.o.g. converges to some limit $(\eta^{G,*},\eta^{H,*})$.
    Reusing the previous sign considerations (note that we never explicitly exploited that $\eta_i^{G,k}$ or $\eta_i^{H,k}$ are assumed to be unbounded to extract the desired signs), we see that $(x^*,\eta^{G,*},\eta^{H,*})$ is a weakly stationary point of \eqref{eq:MPVC}. 
\end{proof}

Note that the assumptions needed to prove this result differ a little from those needed for MPCCs in \cite{KaS13, KaS15}.
For MPCCs convergence of $\eps$-stationary points to a weakly stationary point is proven under MPCC-LICQ (in fact only MPCC-MFCQ is needed) without any additional assumptions in \cite[Theorem 4.13]{KaS13}.
However, since we consider MPVCs here, we have to ensure for both weak stationarity and MPVC-MFCQ that multipliers $\eta^G_i$ associated with gradients $\nabla G_i(x^*)$, $i \in I_{00}$ are nonnegative, whereas there is no sign constraint on multipliers $\eta_i^H$ associated $\nabla H_i(x^*)$, $i \in I_{00}$.
The subsequent example illustrates the need for an additional condition to be able to ensure weak stationarity in the MPVC setting.

\begin{example}\label{exa:LshapedCounterExampleWeak}
    Consider the MPVC
    $$
        \min_{x \in \R^2} f(x) = x_1 - x_2 \ST H(x) = x_2 \geq 0, \quad G(x)H(x) = x_1 x_2 \leq 0.
    $$
    Then $x^* = (0,0)^T$ is feasible but not weakly stationary, because the -- due to MPVC-LICQ -- unique corresponding multipliers are $(\eta^G, \eta^H) = (- 1, -1)$ although $G(x^*) = H(x^*) = 0$.
    
    Now consider the regularized problem $R^{KS}(t)$ for some $t > 0$ and define 
    $$
        x^t := \binom{t^2}{t - t^2}, \quad
        \nu^t := 0, \quad
        \delta^t := \frac{1}{t^2}, \quad
        \eps_t := t^2.
    $$
    Then obviously $x^t \to x^*$ for $t \downarrow 0$ and one easily verifies that $x^t$ with the multipliers $(\nu^t,\delta^t)$ is an $\eps_t$-stationary point of $R^{KS}(t)$ for all $t > 0$.
    Thus, even if $\eps = o(t)$ and MPVC-LICQ holds, the limit point does not have to be weakly stationary.
    Note that $\frac{\eps_t}{G(x^t)} = 1$ for all $t > 0$ and thus the additional assumption \eqref{eq:LepsLimsup} is not satisfied.
\end{example}

Theorem \ref{thm:LconvergenceEpsMFCQ} is sharp in the sense that even under the additional assumptions used, only weak stationarity of the limit can be guaranteed, but not M- or at least T-stationarity.

\begin{example}\label{exa:LshapedCounterExampleTM}
    Consider the MPVC
    $$
        \min_{x \in \R^2} f(x) = -x_1 + x_2 \ST H(x) = x_2 \geq 0, \quad G(x)H(x) = x_1 x_2 \leq 0.
    $$
    Then $x^* = (0,0)^T$ is weakly but not T- or M- stationary, because the -- due to MPVC-LICQ -- unique corresponding multipliers are $(\eta^G, \eta^H) = (1, 1)$ although $G(x^*) = H(x^*) = 0$.
    
    Now consider the regularized problem $R^{KS}(t)$ for some $t > 0$ and define 
    $$
        x^t := \binom{-t^2}{t + t^2}, \quad
        \nu^t := 0, \quad
        \delta^t := \frac{1}{t^2}, \quad
        \eps_t := t^2.
    $$
    Then obviously $x^t \to x^*$ for $t \downarrow 0$ and one easily verifies that $x^t$ with the multipliers $(\nu^t,\delta^t)$ is an $\eps_t$-stationary point of $R^{KS}(t)$ for all $t > 0$.
    Thus, although we have $\eps = o(t)$, MPVC-LICQ, and $\frac{\eps_t}{G(x^t)} = -1$ for all $t > 0$, the limit point $x^*$ is not T-stationary.
\end{example}
\subsection{The Nonsmooth  Regularization} \label{ssec:nonsmoothRegularization}

The following regularization approach is based on a regularization  scheme for complementarity constraints introduced  by Kadrani et al.  \cite{KDB09}.
For $t>0$ it is  given by

\begin{minipage}[c]{0.5\textwidth}
    \begin{align} 
        \min \ & f(x) \nonumber\\ 
        \st \ & g_i(x) \leq 0 \  (i = 1,\ldots,m), \nonumber  \\ 
         & h_j(x) = 0 \ (j = 1,\ldots,p), \tag{$R^{KDB}(t)$} \nonumber \\ 
         & H_i(x) \geq 0 \ (i = 1,\ldots,l),  \nonumber \\ 
         & \Phi^{KDB}_i(x;t) \leq 0 \  (i = 1,\ldots,l),  \nonumber 
    \end{align}
\end{minipage}
\qquad
\begin{minipage}[c]{0.45\textwidth}
    \begin{tikzpicture}[scale = 1.25]
        \fill[blue!30] (-1.5,0.5) -- (0,0.5) -- (0,2) -- (-1.5,2) -- cycle;
        \fill[blue!30]  (0,0) -- (2,0) -- (2,0.5) -- (0,0.5) -- cycle;
        \draw[->] (-1.75,0) -- (2.5,0) node[below left]{$G_i(x)$};
        \draw[->] (0,-0.5)  -- (0,0.5) node[below left]{$t$} -- (0,2.5) node[left]{$H_i(x)$};
        \draw[line width = 1.5pt, blue] (-1.5,0.5) --  (2,0.5);
        \draw[line width = 1.5pt, blue] (0,2) -- (0,00) -- (2,0);
        \draw[->] (0.5,1.5) node[right]{$\Phi^{KDB}_i(x;t)=0$} -- (0,0.75);
        \draw[->] (0.5,1.5)  -- (0.25,0.5);
    \end{tikzpicture}
\end{minipage}

\noindent
with $\Phi_i^{KDB}: \R^n \rightarrow \R$ being defined as
$$
    \Phi_i^{KDB} (x;t) = G_i(x) (H_i(x) - t).
$$
We denote the feasible set of $R^{KDB}(t)$ by  $X^{KDB}(t)$. 

For the analysis of this regularization scheme, we need the following index sets for $t > 0$ and $x \in X^{KDB}(t)$:
\begin{eqnarray*} 
    I_{H}(x) &:=& \{i = 1,\ldots,l \mid H_i(x) = 0\},\\
    I_{\Phi}(x, t) &:=& \{i =1,\ldots,l \mid \Phi_i^{KDB}(x;t) = G_{i}(x)(H_i(x) - t) = 0\},\\
    I_{\Phi}^{0,*}(x, t) &:=& \{i \in I_{\Phi}(x, t) \mid  (H_i(x) - t) = 0\},\\
    I_{\Phi}^{*,0}(x, t) &:=& \{i \in I_{\Phi}(x, t) \mid  G_{i}(x)= 0\}.
\end{eqnarray*}

Since this is the first time that this regularization scheme is applied to MPVCs instead of MPCCs, we start by analyzing the exact case, where we assume that we can compute KKT points of the regularized problems.

\begin{theorem}\label{thm:NonsmoothconvergenceExactCPLD}
     Let $\{t_k\} \downarrow 0$ and let  $\{x^k\}$ be a sequence of KKT points of $R^{KDB}(t_k)$with $x^k\to x^*$ such that MPVC-CPLD holds at $x^*$.
    Further, assume that for all $k$ large
    \begin{equation}\label{eq:NonsmoothConvergenceExactAssumption}
        H_i(x^k) \geq t_k \quad \text{or} \quad G_i(x^k) > 0 \quad \forall i \in I_{00}.
    \end{equation}
    Then $x^*$ is an M-stationary point of \eqref{eq:MPVC}.
\end{theorem}

\begin{proof}
    Since $x^k$ is a KKT point of $R^{KDB}(t_k)$, there exist multipliers $(\nu^k, \delta^k)$ such that 
    \begin{eqnarray}
        &&  0 = \nabla f(x^k) - \sum_{i=1}^l \nu_i^k H_i(x^k) + \sum_{i=1}^l \delta_i^k \nabla \Phi^{KDB}_i(x^k;t_k), \label{eq: KAD1} \\
        && \nu^k \geq 0, \quad \supp(\nu^k) \subseteq I_H(x^k), \\
        && \delta^k \geq 0, \quad \supp(\delta^k) \subseteq I_{\Phi}(x^k,t_k)
    \end{eqnarray}
    for all $k \in \N$.
    Now define the multipliers $(\eta^{G,k}, \eta^{H,k})$ component-wise by
    \begin{equation*}
        \eta_i^{G,k} := \delta_i^k (H_i(x^k) - t_k) 
        \quad \text{and} \quad 
        \eta_i^{H,k} := \nu^k_i -\delta_i^k G_i(x^k)
    \end{equation*}
    for all $i = 1,\ldots, l$.
    Using these definitions, \eqref{eq: KAD1} reads
    \begin{align}\label{eq: KAD2}
        0 = \nabla f(x^k) + \sum_{i=1}^l \eta_i^{G,k}\nabla G_i(x^k) - \sum_{i=1}^l \eta_i^{H,k}\nabla H_i(x^k).
    \end{align}
    By simple continuity arguments, it follows that 
    \begin{equation*}
        \supp(\eta^{G,k}) \subseteq I_{\Phi}^{*0}(x^k,t_k) \subseteq I_{00}  \cup I_{+0}, 
        \quad \text{and} \quad 
        \supp(\eta^{H,k}) \subseteq  I_H(x^k) \cup I_{\Phi}^{0*}(x^k,t_k) \subseteq I_0.
    \end{equation*}
    For all $i \in I_{0-}$ we have $G_i(x^k) < 0$ for all $k$ large and thus $\eta_i^{H,k} \geq 0$.
    For all $i \in I_{+0}$ we have $H_i(x^k) > t_k$ for all $k$ large and thus $\eta^{G,k}_i \geq 0$.
    Finally, for all $i \in I_{00}$, our assumption guarantees that either $H_i(x^k) \geq t_k$ or $\Phi_i(x^k;t_k) <0$ and thus $\delta^k_i = 0$.
    In both cases $\eta^{G,k}_i \geq 0$ follows.

    Using \cite[Lemma A.1]{StU10}, we can assume w.l.o.g. that the gradients
    \begin{equation}\label{eq: kad-indep}
        \nabla G_i(x^k) \; (i \in \supp(\eta^{G,k})) , \quad
        \nabla H_i(x^k) \; (i \in \supp(\eta^{H,k}) )
    \end{equation}
    are linearly independent for all $k \in \N$.
    (Beware that while we can preserve the signs of the multipliers $\eta^{G,k}, \eta^{H,k}$ and the upper estimates for their support, their structure is lost after using this lemma.)
    
    If the sequence $\{(\eta^{G,k},\eta^{H,k})\}$ were unbounded, we can assume w.l.o.g.
    \begin{equation*}
        \frac{(\eta^{G,k},\eta^{H,k})}{\| (\eta^{G,k},\eta^{H,k}) \|} \to (\bar \eta^{G},\bar \eta^{H}) \neq 0.
    \end{equation*}
    Dividing equation \eqref{eq: KAD2} by $\| (\eta^{G,k},\eta^{H,k}) \|$ and taking the limit $k \to \infty$ then yields
    \begin{align}\label{eq: KAD4}
        0 &= \sum_{i \in \supp(\bar \eta^H)} \bar \eta_i^{G} \nabla G_i(x^*) - \sum_{i \in \supp(\bar \eta^G)} \bar \eta_i^{H}\nabla H_i(x^*),
    \end{align}
    where $\supp(\bar \eta^G) \subseteq I_{00} \cup I_{+0}$, $\supp(\bar \eta^H) \subseteq I_{0}$ and  $\bar \eta^H_i \geq 0$ for all $i \in I_{0-}$ as well as $\bar \eta^G_i \geq 0$ for all $i \in I_{+0} \cup I_{00}$.
    Since $(\bar \eta^{G},\bar \eta^{H}) \neq 0$, MPVC-CPLD then implies that the gradients
    \begin{eqnarray*}
        && - \nabla H_i(x^k) \; (i \in \supp(\bar \eta^H) \cap I_{0-}), \quad
        \nabla G_i(x^k) \; (i \in \supp(\bar \eta^G) \cap (I_{+0} \cup I_{00}))\\
        &\text{and}& \nabla H_i(x^k) \; (i \in \supp(\bar \eta^H) \cap (I_{0+} \cup I_{00}))
    \end{eqnarray*}
    remain linearly dependent for all $k$ sufficiently large.
    However, this contradicts our assumption that the gradients in \eqref{eq: kad-indep} are linearly independent for all $k \in \N$.

    Thus, the sequence $\{(\eta^{G,k},\eta^{H,k})\}$ is bounded and we can assume w.l.o.g. that it is convergent to some vector $(\eta^{G,*},\eta^{H,*})$.
    This limit then satisfies
    \begin{equation*}
        0 = \nabla f(x^*)  + \sum_{i=1}^l \eta^{G,*}_i \nabla G_i(x^*) - \sum_{i=1}^l \eta^{H,*}_i H_i(x^*) 
    \end{equation*}
    and, by the same arguments as above,
    \begin{equation*}
        \supp(\eta^{G,*}) \subseteq I_{00} \cup I_{+0} , \quad
        \supp(\eta^{H,*}) \subseteq I_{0}
    \end{equation*}
    as well as $\eta^{H,*}_i \geq 0$ for all $i \in I_{0-}$ and $\eta^{G,*}_i \geq 0$ for all $i \in I_{+0} \cup I_{00}$.
    This shows that $x^*$ is a weakly stationary point of \eqref{eq:MPVC}.

    To prove that the limit is in fact M-stationary, we have to show $\eta^{G,*}_i \eta^{H,*}_i = 0$ for all $i \in I_{00}$.
    Here, we already know $\eta^{G,*}_i \geq 0$ for all $i \in I_{00}$.
    In case $\eta^{G,*}_i > 0$ for some $i \in I_{00}$, we know $\eta^{G,k}_i > 0$ for all $k$ sufficiently large.
    Since applying \cite[Lemma A.1]{StU10} does not enlarge the support of the multipliers and preserves their signs, it follows that $\delta_i^k (H_i(x^k) - t_k) > 0$ for all $k$ sufficiently large and thus $\Phi_i^{KDB}(x^k;t_k) = 0$ and $H_i(x^k) > t_k > 0$.
    But then $G_i(x^k) = 0$ and $\nu_i^k = 0$ and thus $\eta^{H,k_i} = 0$ for all $k$ large.
    This is again preserved under \cite[Lemma A.1]{StU10} and thus implies $\eta^{H,*}_i = 0$.
\end{proof}

In the previous result, we need the additional assumption \eqref{eq:NonsmoothConvergenceExactAssumption} in order to be able to utilize MPVC-CPLD and to prove weak stationarity of the limit.
Such an assumption is not needed for MPCCs, see \cite{KDB09}.
However, the following example illustrates the necessity of such an additional assumption in the MPVC setting.

\begin{example}\label{exa:nonsmoothCounterExampleWeakExact}
    Consider the MPVC
    $$
        \min_{x \in \R^2} f(x) = x_1 \ST H(x) = x_2 \geq 0, \quad G(x)H(x) = x_1 x_2 \leq 0.
    $$
    Then $x^* = (0,0)^T$ is feasible but not weakly stationary, because the -- due to MPVC-LICQ -- unique multipliers are $(\eta^G, \eta^H) = (- 1, 0)$ although $G(x^*) = H(x^*) = 0$.
    
    Now consider the regularized problem $R^{KDB}(t)$ for some $t > 0$ and define 
    $$
        x^t := \binom{0}{\frac{t}{2}}, \quad
        \nu^t := 0, \quad
        \delta^t := \frac{2}{t}
    $$
    Then obviously $x^t \to x^*$ for $t \downarrow 0$ and one easily verifies that $x^t$ with the multipliers $(\nu^t,\delta^t)$ is a KKT point of $R^{KS}(t)$ for all $t > 0$.

    This illustrates that even under MPVC-LICQ, the limit point does not have to be weakly stationary without the additional assumption \eqref{eq:NonsmoothConvergenceExactAssumption}.
    Note that in this example we have $G(x^t) = 0$ and $H(x^t) = \frac{t}{2} < t$ for all $t > 0$ and thus  \eqref{eq:NonsmoothConvergenceExactAssumption} is not satisfied.
\end{example}

Note that a similar situation is not possible when applying the L-shaped regularization from Section \ref{ssec:LshapedRegularization} instead.
This is a new, MPVC-specific observation, because both methods have the same theoretical properties, when applied to MPCCs.

The next step is again to consider how the convergence properties change, when we compute only $\eps$-stationary points of the regularized problems.

\begin{theorem}\label{thm:nonsmoothConvergenceEpsMFCQ}
    Let $\{t_k\}, \{\eps_k\}  \downarrow 0$, and let $\{x^{k}\}$ be a sequence of $\eps_k$-stationary points of $R^{KDB}(t_{k})$ with $x^k \rightarrow x^*$ such MPVC-MFCQ holds at $x^{*}$.
    Assume furthermore that for all $i \in I_{00}$ and all $k$ large
    \begin{equation*}
        [H_i(x^k) \geq t_k \quad \text{or} \quad G_i(x^k) > 0] 
        \quad \text{and} \quad 
        \liminf_{k\to\infty} \frac{\eps_k}{G_i(x^k)}\leq 0,
    \end{equation*}
    whenever the latter is well-defined.
    Then $x^{*}$ is a weakly stationary point of \eqref{eq:MPVC}.
\end{theorem}

\begin{proof}
    Since $x^k$ are $\eps_k$-stationary points of $R^{KDB}(t_k)$, there exist multipliers $(\nu^k,\delta^k)$ such that
    \begin{eqnarray}
        && \left\|\nabla f(x^k)- \sum_{i=1}^k\nu_i^k \nabla H_i(x^k) + \sum_{i=1}^l \delta_i^k \nabla \Phi_i^{KDB}(x^k; t_k)\right\|_\infty \leq \eps_k, \label{eq:nonsmoothEpsStat} \\
        && \nu_i^k \geq -\eps_k, \quad H_i(x^k) \geq -\eps_k, \quad |\nu^k_i H_i(x^k)| \leq \eps_k \quad (i = 1,\ldots, l), \label{eq:nonsmoothEpsStatMult1}\\
        && \delta_i^k \geq -\eps_k, \quad \Phi_i^{KDB}(x^k; t_k) \leq \eps_k, \quad |\delta_i^k \Phi_i^{KDB}(x^k; t_k)| \leq \eps_k \quad (i = 1,\ldots, l). \label{eq:nonsmoothEpsStatMult2}
    \end{eqnarray}
    Here, the gradients of $\Phi_i^{KDB}$ are given by 
    \[
        \nabla \Phi_i^{KDB}(x^k ; t_k) = (H_i(x^k) - t_k) \nabla G_i(x^k) + G_i(x^k)\nabla H_i(x^k) .
    \]
    
    We now define multipliers $\eta^{G,k}, \eta^{H,k} \in \R^l$ component-wise by
    \begin{equation*}
        \eta^{G,k}_i := \delta_i^k (H_i(x^k) - t_k), 
        \quad \text{and} \quad 
        \eta_i^{H,k} := \nu_i^k - \delta_i^k G_i(x^k).
    \end{equation*}
    Using these new multipliers, we can re-write \eqref{eq:nonsmoothEpsStat} as
    \begin{align}\label{eq:nonsmoothEpsStat2}
        \left\|\nabla f(x^k) - \sum_{i=1}^l \eta^{H,k}_i \nabla H_i(x^k) + \sum_{i=1}^l \eta^{G,k}_i \nabla G_i(x^k) \right\|_{\infty} \leq \eps_k. 
    \end{align}
    
    We claim that the sequence $\{(\eta^{G,k}, \eta^{H,k})\}$ is bounded.
    If the sequence were unbounded, then the whole normalized sequence would converge w.l.o.g. 
    \begin{equation*}
        \frac{(\eta^{G,k},\eta^{H,k})}{\|(\eta^{G,k},\eta^{H,k})\|} \to (\bar{\eta}^G, \bar{\eta}^H) \neq 0.
    \end{equation*}
    Then, dividing \eqref{eq:nonsmoothEpsStat2} by $\|(\eta^{G,k},\eta^{H,k})\|$ and passing to the limit we obtain
    \begin{align}\label{eq:nonsmoothEpsLinDep}
        0 = \sum_{i=1}^l \bar{\eta}_i^G \nabla G_i(x^{*}) - \sum_{i=1}^l \bar{\eta}^{H}_i\nabla H_i(x^{*}).    
    \end{align}
    Here, for all $i \in I_{+-}$, we have $\lim_{k\to \infty} \Phi_i^{KDB}(x^k;t_k) \neq 0$, and hence $\delta_i^k\to 0$ due to  \eqref{eq:nonsmoothEpsStatMult2}.
    Therefore, $\eta_i^{G,k}\to 0$ and hence $\bar{\eta}^G_i = 0$.
    For all  $i \in I_{0+} \cup I_{0-}$ we have 
    \[
        G_i(x^k)\eta_i^{G,k} = G_i(x^k) \delta_i^k (H_i(x^k) - t_k) = \delta_i^k \Phi_i^{KDB}(x^k;t_k) \to  0, 
    \]
    Since $G_i(x^*) \not \to0$, this implies  $\eta_i^{G,k}\to 0$ and hence $\bar{\eta}^G_i = 0$.
    All in all, we have proven $\supp(\bar{\eta}^{G}) \subseteq I_{00} \cup I_{+0}$. 
    
    Now, for all $i \in I_{+}$, we know $\nu_i^k \to 0$ by \eqref{eq:nonsmoothEpsStatMult1} and thus using \eqref{eq:nonsmoothEpsStatMult2} can conclude
    \[ 
        \eta_i^{H,k}(H_i(x^k) - t_k) = (\nu_i^k - \delta_i^k G_i(x^k))(H_i(x^k) - t_k) = \nu_i^k(H_i(x^k) - t_k)- \delta_i^k \Phi_i^{KDB}(x^k; t_k)\to 0.
    \] 
    Here, $H_i(x^*) >0 $ implies $\eta_i^{H,k}\to 0$.
    This shows $\supp(\bar{\eta}^H) \subseteq I_0 = I_{0-} \cup I_{0+} \cup I_{00}$.
    Hence, \eqref{eq:nonsmoothEpsLinDep} reduces to
    \begin{align*}
        0 = \sum_{i\in I_{00} \cup I_{+0}} \bar{\eta}_i^G \nabla G_i(x^*) - \sum_{i \in I_{00} \cup I_{0-} \cup I_{0+}} \bar{\eta}^{H}_i \nabla H_i(x^*).  
    \end{align*}
    
    The next step is to determine the signs of the multipliers $\bar \eta_i^G\;(i\in I_{00}\cup I_{+0})$ and $\bar \eta_i^H\;(i\in I_{0-})$.
    To this end, first consider $i\in I_{0-}$ and $\eta_i^{H,k} = \nu_i^k - \delta_i^k G_i(x^k)$.
    Due to $\nu_i^k \geq -\eps_k \to 0$ and $\delta_i^k \geq - \eps_k \to 0$ together with $-G_i(x^k) > c > 0$ for all $k$ large we can conclude 
    \begin{equation*}
        \liminf_{k \to \infty} \eta_i^{H,k} = \liminf_{k \to \infty} \nu_i^k - \delta_i^k G_i(x^k) \geq 0
        \quad \Longrightarrow \quad
        \bar \eta^{H}_i \geq 0.
    \end{equation*}
    
    Now consider $i\in I_{+0}$ and $\eta^{G,k}_i = \delta_i^k (H_i(x^k) - t_k)$.
    Since $\delta_i^k \geq - \eps_k \to 0$ and $H_i(x^k) - t_k > c > 0$ for all $k$ large, this implies
    \begin{equation*}
        \liminf_{k \to \infty} \eta_i^{G,k} = \liminf_{k \to \infty} \delta_i^k (H_i(x^k) - t_k) \geq 0
        \quad \Longrightarrow \quad
        \bar \eta^{G}_i \geq 0.
    \end{equation*}
    
    Finally, consider $i \in I_{00}$ and assume that $\eta^{G,k}_i = \delta_i^k (H_i(x^k) - t_k) < -c < 0$ for all $k$ large.
    Since $H_i(x^k) - t_k \to 0$ and $\delta^k_i \geq -\eps_k$, this implies $\delta^k_i \to + \infty$ and $H_i(x^k) < t_k $ for all $k$ large.
    Our assumption then implies $G_i(x^k) > 0$ and thus
    \begin{equation*}
        0 < - \eta^{G,k}_i G_i(x^k) = -\delta^k_i (H_i(x^k) - t_k) G_i(x^k) =|\delta^k_i \Phi_i^{KDB}(x^k;t_k)| \leq \eps_k
    \end{equation*}
    for all $k$ large.
    But then it follows that
    \begin{equation*}
        \frac{\eps_k}{G_i(x^k)} \geq - \eta^{G,k}_i > c > 0
    \end{equation*}
    for all $k$ large in contradiction to the assumptions.
    This shows $\bar \eta^{G}_i \geq 0$ for all $I \in I_{00}$.
    
    Since we have verified $\bar \eta_i^G\geq 0 \; (i\in I_{00}\cup I_{+0})$ and $\bar \eta_i^H \geq 0 \; (i\in I_{0-})$, equation \eqref{eq:nonsmoothEpsLinDep} and $(\bar \eta^G, \bar \eta^H) \neq (0,0)$ contradict the assumption that MPVC-MFCQ holds in $x^*$.
    
    Therefore the sequence $\{(\eta^{G,k}, \eta^{H,k})\}$ is bounded and w.l.o.g. converges to some limit $(\eta^{G,*},\eta^{H,*})$.
    Reusing the previous sign considerations (note that we never needed $|\eta_i^{G,k}| \to \infty$ or $|\eta_i^{H,k}| \to \infty$ to extract the desired supports and signs), we see that $(x^*,\eta^{G,*},\eta^{H,*})$ is a weakly stationary point of \eqref{eq:MPVC}. 
\end{proof}

Thus, similarly to the local and the L-shaped regularization, the favourable theoretical convergence properties of this regularization scheme are lost in the inexact case.
Examples \ref{exa:nonsmoothCounterExampleWeakExact} and \ref{exa:LshapedCounterExampleWeak} illustrate that neither of the two additional assumptions for $i \in I_{00}$ can be dropped.
To see this, note that in Example \ref{exa:LshapedCounterExampleWeak} we have $x_1^t + x_2^t = t$ for all $t > 0$ and thus the L-shaped regularization coincides with the nonsmooth regularization in $x^t$.
Example \ref{exa:LshapedCounterExampleTM}, where again $x_1^t + x_2^t = t$ for all $t > 0$, illustrates that we cannot guarantee more than weak stationarity of the limit under the assumptions of Theorem \ref{thm:nonsmoothConvergenceEpsMFCQ}.

We finish our discussion of the nonsmooth regularization method by proving that the regularized problems locally satisfy standard GCQ.

\begin{theorem}\label{thm:nonsmoothStandardCQGCQ}
    Let $x^*$ be feasible for the MPVC \eqref{eq:MPVC} such that MPVC-LICQ holds at $x^*$.
    Then there exists $\bar{t} > 0$ and a neighborhood $N(x^*)$ of $x^*$ such that, for all $t \in (0, \bar{t}]$ and all $x \in X^{KDB}(t) \cap N(x^*)$, standard GCQ for $R^{KDB}(t)$ is satisfied at $x$.
\end{theorem}

\begin{proof}
    Let $N(x^*)$ be a neighborhood of $x^*$ and a $\bar t > 0$.
    Now consider an arbitrary $t \in (0,\bar t]$ and  $\hat x \in N(x^*) \cap X^{KDB}(t)$ and define the index sets
    \begin{eqnarray*}
        I_\Phi^{0-} &:=& \{i \in I_\Phi(\hat x,t) \mid H_i(\hat x) = t, \; G_i(\hat x) < 0 \}, \\
        I_\Phi^{-0} &:=& \{i \in I_\Phi(\hat x,t) \mid H_i(\hat x) < t, \; G_i(\hat x) = 0 \}, \\
        I_\Phi^{00} &:=& \{i \in I_\Phi(\hat x,t) \mid H_i(\hat x) = t, \; G_i(\hat x) = 0 \}, \\
        I_\Phi^{+0} &:=&  \{i \in I_\Phi(\hat x,t) \mid H_i(\hat x) > t, \; G_i(\hat x) = 0 \},\\
        I_\Phi^{0+} &:=&  \{i \in I_\Phi(\hat x,t) \mid H_i(\hat x) = t, \; G_i(\hat x) > 0 \},\\
        J^{+-} &:=& \{i =1,\ldots,l \mid H_i(\hat x) > t, \; G_i(\hat x) < 0 \},\\
        J^{-+} &:=& \{i =1,\ldots,l \mid 0 < H_i(\hat x) < t, \; G_i(\hat x) > 0 \}.
    \end{eqnarray*}
    The definition of these index sets implies 
    \begin{eqnarray*}
         I_\Phi^{0-} \cup I_\Phi^{-0} \cup I_\Phi^{00} \cup I_\Phi^{+0} \cup I_\Phi^{0+} &=& I_\Phi(\hat x,t), \\
        (J^{+-} \cup J^{-+}) \cap (I_\Phi(\hat x,t) \cup I_H(\hat x)) &= & \emptyset.
    \end{eqnarray*}
    By choosing $N(x^*)$ and $\bar t > 0$ sufficiently small, we can guarantee the inclusions
    \begin{eqnarray*}
        I_\Phi^{0-} \cup I_\Phi^{00} \cup I_\Phi^{0+} \cup I_H(\hat x) &\subseteq& I_{0}, \\
        I_\Phi^{-0} \cup I_\Phi^{00} \cup I_\Phi^{+0}  &\subseteq& I_{00} \cup I_{+0},
    \end{eqnarray*}
    and, due to MPVC-LICQ, that the following gradients are linearly independent:
    \begin{equation*}
        \nabla H_i(\hat x) \; (i \in I_\Phi^{0-} \cup I_\Phi^{00}  \cup I_\Phi^{0+} \cup I_H(\hat x) ), \quad
        \nabla G_i(\hat x) \; (i \in I_\Phi^{-0} \cup I_\Phi^{00}  \cup I_\Phi^{+0}).
    \end{equation*}

    For all  $J \subseteq I_\Phi^{+0}$, we define the auxiliary problem NLP$(\hat x,J)$  as
    \begin{eqnarray*}
       \min f(x) & \st & H_i(x) \geq t, \; G_i(x) \leq 0 \quad \forall i \in I_\Phi^{0-} \cup J  \cup I_\Phi^{+0} \cup J_{+-} \\
       && 0 \leq H_i(x) \leq t, \; G_i(x) \geq 0 \quad  \forall i \in I_\Phi^{-0} \cup (I_\Phi^{00} \setminus J)  \cup I_\Phi^{0+} \cup I_H(\hat x) \cup J_{-+}
    \end{eqnarray*}
    and denote its feasible set by $X(J)$.
    Then, $\hat{x} \in X(J)$ and $X(J) \subseteq X^{KDB}(t)$.
    The latter implies
    \begin{equation*}
        \bigcup_{J \subseteq I_\Phi^{+0}} T_{X(J)}(\hat x) \subseteq T_{X^{KDB}(t)}(\hat x).
    \end{equation*}
    To see the opposite inclusion, consider an arbitrary $d \in T_{X^{KDB}(t)}(\hat x)$ and let $\{x^k\} \subseteq X^{KDB}(t)$ and $\{\tau_k\} \geq 0$ with $x^k \to x^*$ and $\tau_k(x^k-x^*) \to d$.
    Then for all $i \in I_\Phi^{0-} \cup I_\Phi^{+0} \cup J^{+-}$ it is easy to see that $H_i(x^k) \geq t$ and $G_i(x^k) \leq 0$ for all $k$ large.
    Analogously, for all $i \in I_\Phi^{-0} \cup I_\Phi^{0+} \cup I_H(\hat x) \cup J^{-+}$ we have $H_i(x^k) \in  [0,t]$ and $G_i(x^k) \geq 0$ for all $k$ large.
    For all $i \in I_\Phi^{00}$ either one of the two conditions can be satisfied in $x^k$.
    But since $I_\Phi^{00}$ is finite, we can assume w.l.o.g. that for some $J \subseteq I_\Phi^{00}$ we have
    \begin{eqnarray*}
        H_i(x^k) \geq t, \; G_i(x^k) \leq 0 && \forall i \in J \\
        0 \leq H_i(x^k) \leq t, \; G_i(x^k) \geq 0 &&  \forall i \in I_\Phi^{00} \setminus J
    \end{eqnarray*}
    for all $k$ large.
    This implies $x^k \in X(J)$ for all $k$ large and thus $d \in T_{X(J)}(\hat x)$.
    We thus know
    \begin{equation}\label{eq:nonsmoothTangentCone}
        T_{X^{KDB}(t)}(\hat x) = \bigcup_{J \subseteq I_\Phi^{00}} T_{X(J)}(\hat x).
    \end{equation}
    
    To compute $T_{X(J)}(\hat x)$ for $J \subseteq I_\Phi^{00}$, note that the active gradients for NLP($J$) in $\hat x$ are
    \begin{eqnarray*}
        \nabla H_i(\hat x) && \forall i \in I_\Phi^{0-} \cup I_\Phi^{00} \cup I_\Phi^{0+} \cup I_H(\hat x), \\
        \nabla G_i(\hat x) && \forall i \in I_\Phi^{-0} \cup I_\Phi^{00} \cup I_\Phi^{+0}.
    \end{eqnarray*}
    These are linearly independent by choice of $N(x^*)$ and $\bar t$ and thus LICQ and ACQ for NLP($\hat x,J$) hold in $\hat x$.
    Using \cite[Theorem~3.1.9]{BaS12} we thus obtain
    \begin{equation*}
        T_{X^{KDB}(t)}(\hat x)^\circ = \bigcap_{J \subseteq I_\Phi^{00}} T_{X(J)}(\hat x)^\circ = \bigcap_{J \subseteq I_\Phi^{00}} L_{X(J)}(\hat x)^\circ,
    \end{equation*}
    where for $J \subseteq I_\Phi^{00}$ the polar of the linearization cones $L_{X(J)}(\hat x)$ is given by
    \begin{eqnarray*}
        L_{X(J)}(\hat x)^\circ = \Big\{ & s = & - \sum_{i \in  I_\Phi^{0-} \cup I_\Phi^{00} \cup I_\Phi^{0+} \cup I_H(\hat x)} \eta^H_i \nabla H_i(\hat x) + \sum_{i \in I_\Phi^{-0} \cup I_\Phi^{00} \cup I_\Phi^{+0}} \eta^G_i \nabla G_i(\hat x) \mid \\
        && \eta^H_{I_\Phi^{0-} \cup J \cup I_H(\hat x)} \geq 0, \quad
        \eta^H_{(I_\Phi^{00} \setminus J) \cup I_\Phi^{0+}} \leq 0, \quad
        \eta^G_{I_\Phi^{+0} \cup J} \geq 0, \quad
        \eta^G_{(I_\Phi^{00} \setminus J) \cup I_\Phi^{-0}} \leq 0 \Big\}
    \end{eqnarray*}
    according to \cite[Theorem~3.2.2]{BaS12}.
    To complete the proof, it remains to show
    \begin{equation*}
        T_{X^{KDB}(t)} (\hat x)^\circ \subseteq  L_{X^{KDB}(t)} (\hat x)^\circ.
    \end{equation*}
    Here, the polar of the linearization cone $L_{X^{KDB}(t)} (\hat x)$ is given by
    \begin{eqnarray*}
        L_{X^{KDB}(t)}(\hat x)^\circ = \Big\{ s \in \R^n \mid& s =& - \sum_{i \in I_H(\hat x)} \nu_i \nabla H_i(\hat x) + \sum_{i \in I_\Phi(\hat x,t)} \delta_i  G_i(\hat x) \nabla H_i(\hat x) \\
        && + \sum_{i \in I_\Phi(\hat x,t)} \delta_i (H_i(\hat x)-t) \nabla G_i(\hat x), \quad
        \nu \geq 0, \quad
        \delta \geq 0 \Big\}
    \end{eqnarray*}
    To this end, consider an arbitrary $s \in T_{X^{KDB}(t)} (\hat x)^\circ$.
    Then by our previous considerations, we know that $s \in L_{X(J)}(\hat x)^\circ$ for all $J \subseteq I_\Phi^{00}$.
    Due to the linear independence of the gradients, this implies that the representation of $s$ does not depend on $\nabla G_i(\hat x), \nabla H_i(\hat x)$ with $i \in I_\Phi^{00}$ and thus
    \begin{eqnarray*}
        && s = - \sum_{i \in  I_\Phi^{0-} \cup I_\Phi^{0+} \cup I_H(\hat x)} \eta^H_i \nabla H_i(\hat x) + \sum_{i \in I_\Phi^{-0} \cup I_\Phi^{+0}} \eta^G_i \nabla G_i(\hat x)\\
        \text{with} && \eta^H_{I_\Phi^{0-} \cup I_H(\hat x)} \geq 0, \quad
        \eta^H_{ I_\Phi^{0+}} \leq 0, \quad
        \eta^G_{I_\Phi^{+0} } \geq 0, \quad
        \eta^G_{ I_\Phi^{-0}} \leq 0.
    \end{eqnarray*}
    This shows $s \in L_{X^{KDB}(t)}(\hat x)^\circ$ and completes the proof of GCQ for $R^{KDB}(t)$ in $\hat x$.
\end{proof}

\subsection{Theoretical Comparison of all four Regularization Schemes}\label{ssec:theoreticalComparison}

Our analysis of the theoretical properties of these four regularization schemes allows us to compare them with respect to  limit of KKT points, the limit of $\eps$-stationary points,  and the regularity of their feasible sets.

We begin with the limit of KKT points of the regularized problems, see Table \ref{tab:comparisonExact}.
In this regard, the L-shaped regularization is the clear victor, because it guarantees the strongest stationarity of $x^*$ under the weakest assumptions.

\begin{table}[htb]
    \caption{Comparison of the limit $x^*$ of KKT points} \label{tab:comparisonExact}
    \centering
    \begin{tabular}{llll}
        \toprule
        regularization & CQ in $x^*$ & additional assumptions & stationarity of $x^*$ \\
        \midrule
         global & MPVC-MFCQ & none & T-stationarity \\
         local & MPVC-CPLD & none & T-stationarity \\
         L-shaped & MPVC-CPLD & none & M-stationarity \\
         nonsmooth & MPVC-CPLD & yes\footnotemark & M-stationarity\\
         \bottomrule
    \end{tabular}
\end{table}

However, when we instead consider the limit of $\eps$-stationary points, the picture changes, see Table \ref{tab:comparisonInexact}.
The only regularization, which guarantees more than weak stationarity in this setting, is the global regularization.
Furthermore, this regularization does not require additional assumptions.

\begin{table}[htb]
    \caption{Comparison of the limit $x^*$ of $\eps$-stationary points} \label{tab:comparisonInexact}
    \centering
    \begin{tabular}{llll}
        \toprule
        regularization & CQ in $x^*$ & additional assumptions & stationarity of $x^*$ \\
        \midrule
         global & MPVC-MFCQ & none & T-stationarity \\
         local & MPVC-MFCQ & none & weak stationarity \\
         L-shaped & MPVC-MFCQ & yes\footnotemark & weak stationarity \\
         nonsmooth & MPVC-MFCQ & yes\footnotemark & weak stationarity\\
         \bottomrule
    \end{tabular}
\end{table}

With regard to the the regularity of the resulting feasible sets, see Table \ref{tab:comparisonRegularity}, the global regularization again looks very promising.
However, a closer look reveals that both the L-shaped and the nonsmooth regularization also satisfy LICQ in all feasible points $x$ with $(G_i(x),H_i(x)) \neq (0,t)$ for all $i=1,\ldots,l$.
In contrast, the feasible set of the local regularization can satisfy LICQ only in points $x$, where $(G_i(x),H_i(x)) \approx (0,0)$ for all $i=1,\ldots,l$.

\begin{table}[htb]
    \caption{Comparison of the local regularity of the feasible sets} \label{tab:comparisonRegularity}
    \centering
    \begin{tabular}{llc}
        \toprule
        regularization & MVPC-CQ at $x^*$ & standard CQ around $x^*$ \\
        \midrule
         global & MPVC-MFCQ & MFCQ \\
         local & MPVC-LICQ & ACQ \\
         L-shaped & MPVC-LICQ & GCQ \\
         nonsmooth & MPVC-LICQ & GCQ\\
         \bottomrule
    \end{tabular}
\end{table}

\addtocounter{footnote}{-2}\footnotetext{$\forall i \in I_{00}$: $H_i(x^k) \geq t_k$ or $G_i(x^k) > 0$.}
\addtocounter{footnote}{+1}\footnotetext{$\forall i \in I_{00}$: $\liminf_{k \to \infty} \frac{\eps_k}{G_i(x^k)} \leq 0$.}
\addtocounter{footnote}{+1}\footnotetext{$\forall i \in I_{00}$: $H_i(x^k) \geq t_k$ or $G_i(x^k) > 0$ and $\liminf_{k \to \infty} \frac{\eps_k}{G_i(x^k)} \leq 0$.}
\section{Numerical Comparison}\label{sec:numerical}

In this chapter, we test and compare the presented regularization strategies numerically.
We consider problems arising from truss topology optimization and optimal control of aircraft trajectories.
Before presenting these models and the numerical results, let us describe the procedure we followed in order to test the numerical behavior of the regularization methods, whose theoretical properties we discussed before.

\begin{algorithm}[htb]
    \begin{algorithmic}[1]
        \REQUIRE An initial point $x^0 \in \R^n$, an initial regularization parameter $t_0 > 0$, a reduction parameter $\sigma \in (0,1)$, a minimum regularization parameter $t_{\min} \in (0, t_0)$, and a feasibility tolerance $\text{tol} > 0$.
        Set $k := 0$.
        \WHILE{$t_k\geq t_{\min}$ and $\text{maxVio} (x^k) > \text{tol}$}
            \STATE Compute a solution $x^{k+1}$ of the regularized problem $R(t_k)$ using $x^k$ as initial point.
            
            \STATE Decrease the regularization parameter $t_{k+1} := t_k \cdot \sigma$ and update $k \gets k+1$.
        \ENDWHILE
        \ENSURE The final iterate $x^* := x^k$ and the corresponding function value $f(x^*)$.
    \end{algorithmic}
    \caption{Abstract regularization algorithm}\label{alg:regularization}
\end{algorithm}

We used Algorithm \ref{alg:regularization} for all the test examples and regularization schemes in order to ensure the methods are tested under the same conditions.
It was implemented in MATLAB using the NLP solver \texttt{fmincon} with the \texttt{SQP} option to solve $R(t_k)$ on each iteration.
The maximum constraint violation in a point $x$ is defined as
$$
   \text{maxVio}(x) = \max_{i=1,...,l}\{G_i(x) H_i(x)\},
$$
and is used to ensure that the vanishing constraints are fulfilled.
The parameters for the algorithm were chosen as
\begin{equation*}
    t_0 = 1, \quad
    \sigma = 0.1, \quad 
    t_{\min} = 10^{-8}, \quad
    \text{tol} = 10^{-6}.
\end{equation*}
Thus, the algorithm terminates either when $x^k$ is sufficiently feasible, i.e. $\text{maxVio}(x^k) \leq 10^{-6}$, or when the regularization parameter becomes too small, i.e. $t_k < 10^{-8}$.
In the latter case, the regularized problem is numerically almost identical to the original MPVC and a further decrease of $t_k$ is not beneficial.

For the local regularization, we used the regularization function
\begin{equation*}
   \theta (x) := \frac{2}{\pi} \sin \left( \frac{\pi}{2} x + \frac{3 \pi}{2} \right) + 1 .
\end{equation*}

In the subsequent sections, the results are given in terms of objective function value $f(x^*)$, the maximum constraint violation of \emph{all} constraint functions in $x^*$ and the number of regularization iterations executed.
They are also compared to the result achieved by applying the NLP solver \texttt{fmincon} directly to the MPVC.

\subsection{Academic Example}\label{ssec:acdemicTruss}

To show the positive influence of regularization methods, we first consider the following two-dimensional academic truss topology optimization problem taken from \cite{ChG97}:

\begin{minipage}[c]{0.55\textwidth}
    \begin{eqnarray}
       \min_{x \in \R^2} && f(x) = 4x_1+2x_2  \nonumber\\
       \st && H_1(x) = x_1 \geq 0, \nonumber\\
       && H_2(x) = x_2 \geq 0, \label{eq:academicTruss}\\
       && G_1(x)H_1(x) = (5\sqrt{2}-x_1-x_2)x_1 \leq 0, \nonumber\\
       && G_2(x)H_2(x) = (5-x_1-x_2)x_2 \leq 0 .\nonumber
    \end{eqnarray}
\end{minipage}
\qquad
\begin{minipage}[c]{0.44\textwidth}
    \begin{tikzpicture}[scale = 0.5]
        \fill[blue!30] ({sqrt{50}},0) -- (8,0) -- (8,8) -- (0,8) -- (0,{sqrt{50}}) -- cycle;
        
        \draw[->] (-1,0) -- (0,0) node[below left]{$0$} -- ({sqrt(50)},0) node[below]{$5\sqrt{2}$} -- (8.5,0) node[below]{$x_1$};
        \draw[->] (0,-1) -- (0,5) node[left]{$5$} -- (0,{sqrt(50)}) node[left]{$5\sqrt{2}$} -- (0,8.5) node[left]{$x_2$};
        
        \draw[, line width = 1.5pt, blue] (0,8) -- (0,{sqrt{50}}) --  ({sqrt{50}},0) -- (8,0);
        \draw[- , line width = 1.5pt, blue] (0,{sqrt{50}}) -- (0,5);
        
        \node[circle, fill=blue, inner sep=2pt, minimum size=3pt,label=above right:{$x^\circ$}] at (0,0) {};
        \node[circle, fill=red, inner sep=2pt, minimum size=3pt,label= right:{$x^*$}] at (0,5) {};
        \node[circle, fill=black, inner sep=2pt, minimum size=3pt,label=right:{$x^+$}] at (0,{sqrt(50)}) {};
    \end{tikzpicture}
\end{minipage}

Here, the weight of a truss consisting of four bars shall be minimized and the variables $x_1$, $x_2$ represent the cross-sectional areas of two different groups of bars.
The feasible set of this problem consists of an unbounded polyhedron with the attached line segment $ \{0\} \times [5, 5\sqrt{2}]$ and the isolated point $(0,0)^T$.

The origin $x^\circ = (0,0)^T$ is the global minimizer of the problem, and $x^* = (0,5)^T$ is a local minimizer.
Additionally, $x^\circ, x^*$ are the only M-stationary points of this MPVC.
However, geometry indicates that numerical methods may also converge to $x^+ =(0,5\sqrt{2})^T$, which is a weakly stationary point but not a local minimizer.
To illustrate this behavior, we chose a grid of $676$ initial points $x^0$ in $[-5, 20]\times[-5,20]$ and attempted to solve \eqref{eq:academicTruss} from those using the four regularization schemes as well as applying \texttt{fmincon} directly to the MPVC. 

\begin{figure}
    \centering
    \begin{subfigure}[b]{.32\linewidth}
        \centering
        \includegraphics[width=\linewidth]{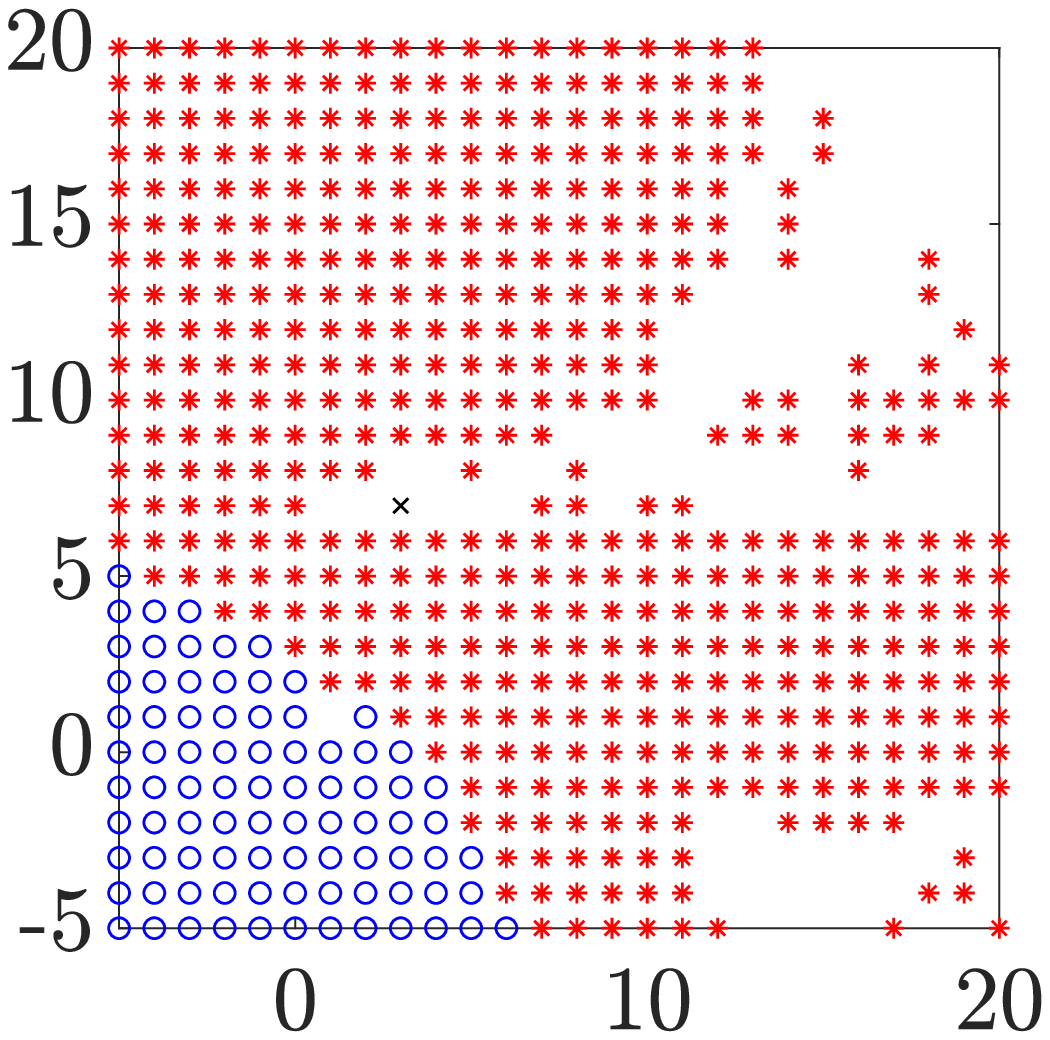}
        \caption{no regularization}\label{fig:diracademic}
    \end{subfigure}
    \begin{subfigure}[b]{.32\linewidth}
        \centering
        \includegraphics[width=\linewidth]{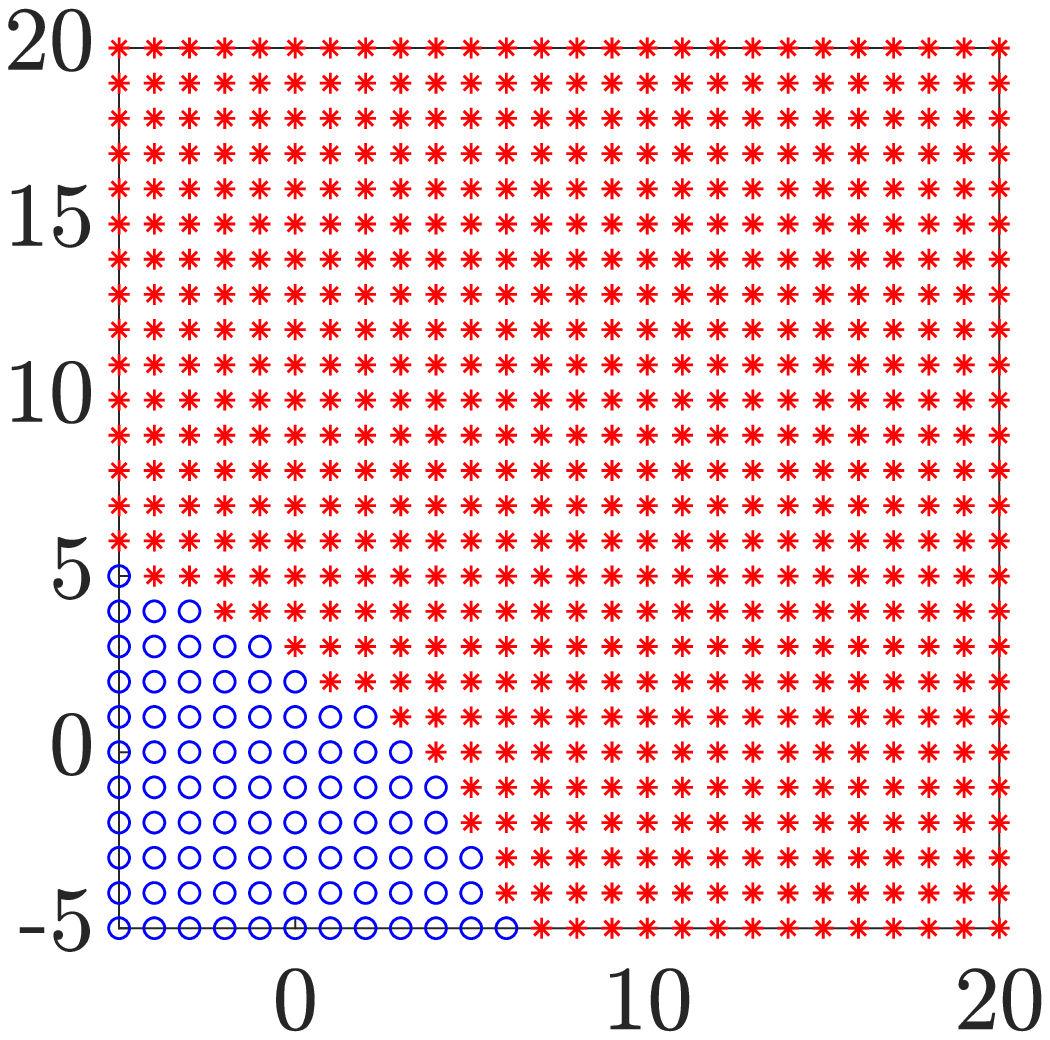}
        \caption{global regularization}\label{fig:schoacademic}
    \end{subfigure}
    \begin{subfigure}[b]{.3\linewidth}
        \centering
        \includegraphics[width=\linewidth]{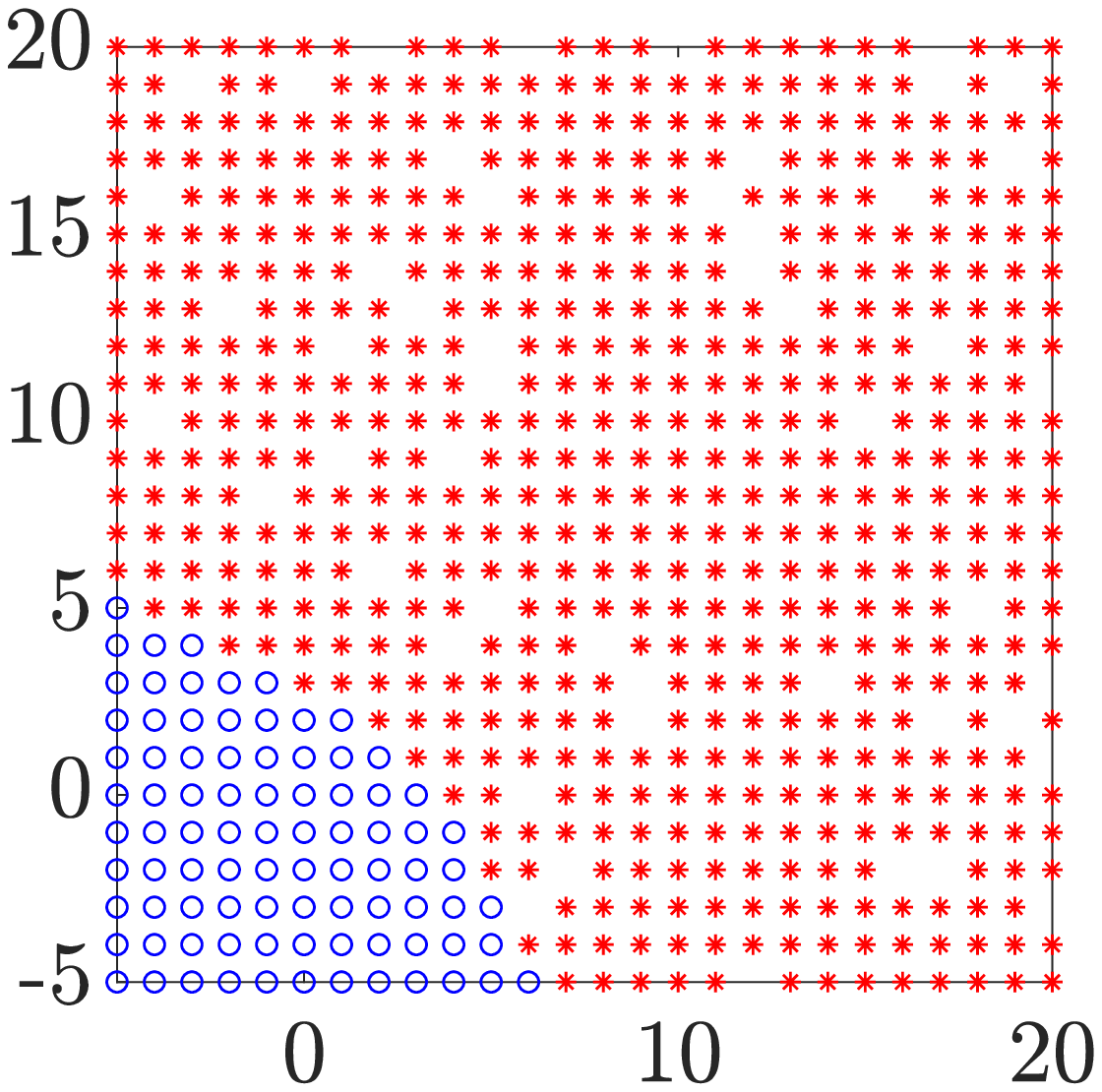}
        \caption{local regularization}\label{fig:steffacademic}
    \end{subfigure}
    \\
    \begin{subfigure}[b]{.32\linewidth}
        \centering
        \includegraphics[width=\linewidth]{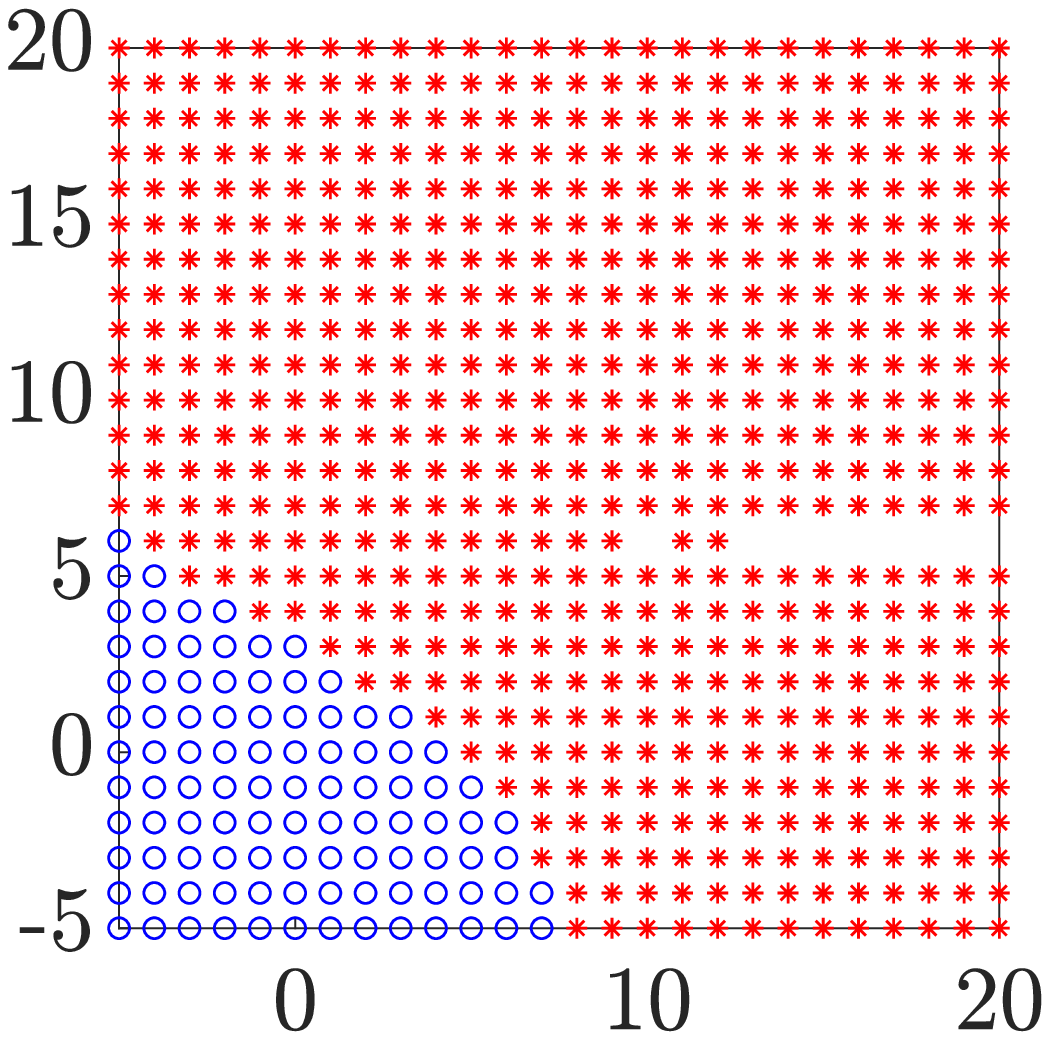}
        \caption{L-shaped regularization}\label{fig:schwacademic}
    \end{subfigure}
    \begin{subfigure}[b]{.3\linewidth}
        \centering
        \includegraphics[width=\linewidth]{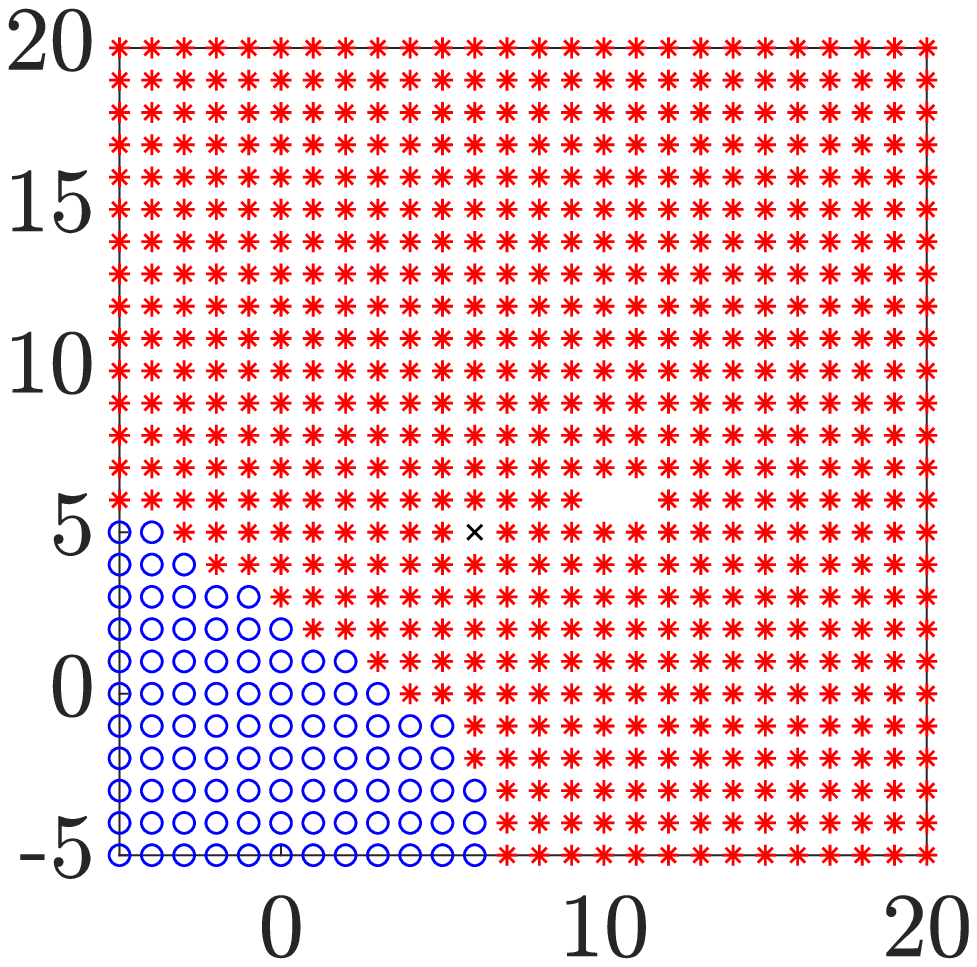}
        \caption{nonsmooth regularization}\label{fig:kadracademic}
    \end{subfigure}
    \caption{Results for the academic example \eqref{eq:academicTruss}} \label{fig:academicexample}
\end{figure}

The results are given in Figure \ref{fig:academicexample}, where we marked an initial point with $\circ$, if the solution found differed less than $10^{-3}$ from $x^\circ$, with $*$, if the solution found was close to $x^*$, and with $+$, if the solution found was close to $x^+$.
The added total number of iterations needed to solve the problem from each initial point and the number of times the algorithm reached a solution close to $x^{\circ}$, $x^{*}$ or $x^{+}$ are gathered in Table \ref{tab:perfacademic}.

\begin{table}
    \centering
    \begin{tabular}{lccccc}
        \toprule
        method   & total iterations & $x^\circ$ & $x^*$ & $x^+$ & neither \\
        \midrule
        no regularization         &  676 &  85 & 453 & 1 & 137 \\
        global regularization     & 5482 &  86 & 590 & 0 &   0  \\
        local regularization      & 1353 &  87 & 539 & 0 &   50   \\
        L-shaped regularization   & 1352 & 100 & 567 & 0 &   9  \\
        nonsmooth regularization  & 1368 &  91 & 582 & 1 &   2  \\
        \bottomrule
    \end{tabular}
    \caption{Performance results for the academic example \eqref{eq:academicTruss}} \label{tab:perfacademic}
\end{table}

Applying \texttt{fmincon} directly without a regularization failed to recover one of the three points of interest for approximately 20\% of the initial points.
In contrast, the regularization methods found the global optimum $x^\circ$ or the local optimum $x^*$ for almost all initial values.
The global regularization succeeded for all initial values but also needed about four times as many iterations as the other methods.
The L-shaped and nonsmooth regularization performed similarly, whereas the local regularization failed to find one of the three points of interest for more initial points.
Interestingly, all five methods rarely terminated in the weakly stationary point $x^+$.

\subsection{Ten-bar Truss}\label{ssec:tenBarTruss}

Our next example is the ``ten-bar truss'', a well-known problem in the engineering literature, see, e.g., \cite{Kir90, ChG97, AHK13} for a more detailed background.
We consider a $2 \times 1$ sized truss consisting of six nodes and ten potential bars, where the two nodes on the left side are fixed (e.g. on a wall) and a force $f$ with $\|f\| = 1$ pulls down on the bottom right hand node.
The ground structure is depicted in Figure \ref{fig:tenbarGround}.

\begin{figure}
    \begin{subfigure}[b]{.32\linewidth}
        \centering
        \includegraphics[width=\linewidth]{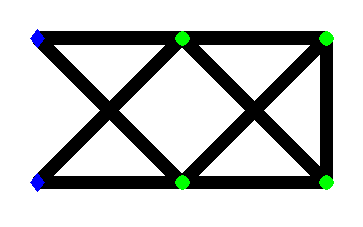}
        \caption{ground structure}\label{fig:tenbarGround}
    \end{subfigure}
    \begin{subfigure}[b]{.32\linewidth}
        \centering
        \includegraphics[width=\linewidth]{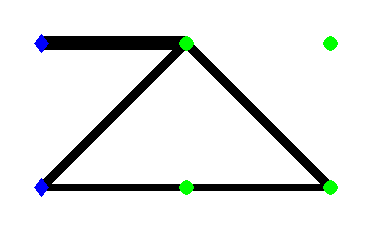}
        \caption{optimal truss}\label{fig:tenbarOpt}
    \end{subfigure}
    \begin{subfigure}[b]{.32\linewidth}
        \centering
        \includegraphics[width=\linewidth]{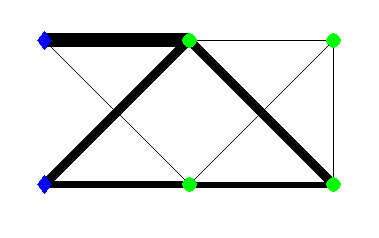}
        \caption{nonsmooth solution}\label{fig:tenbarOptKad}
    \end{subfigure}
    \caption{Results for ten-bar truss example \eqref{eq:tenbar}}\label{fig:tenbar}
\end{figure}

The aim is to minimize the weight of the truss.
However, if we  assume that the same material is used for all bars, we can minimize the volume instead.
To write down the corresponding optimization problem, we need some notation first.
For all bars $i=1,\ldots,10$, we denote the cross-sectional area of bar $i$ by $a_i$ and its length by $l_i$. 
Additionally to the variables $a_i$ ($i=1,\ldots,10$) we have eight auxiliary variables $u_j$ ($j=1,\ldots,8$) representing the nodal displacement of the four free nodes in $x$- and $y$-direction.
With this notation, we obtain the following optimization problem
\begin{eqnarray}
   \min_{a \in \R^{10}, u \in \R^8} \sum_{i=1}^{10}{l_i a_i} & \st & K(a)u = f, \nonumber\\
   && f^Tu \leq c, \nonumber\\
   && 0 \leq a_i \leq \bar{a} \quad \forall i=1,\ldots,10, \label{eq:tenbar}\\
   && (\sigma_i(a,u)^2-\bar{\sigma}^2)a_i \leq 0 \quad \forall i=1,\ldots, 10.\nonumber
\end{eqnarray}

Here, the matrix $K(a)$ is the global stiffness matrix of the truss and is given by
$$
   K(a) = \sum_{i=1}^{10}{a_i \frac{E}{l_i}\gamma_i \gamma_i^T} \in \R^{8\times8}
$$
with some vectors $\gamma_i \in \R^8$ and Young's modulus $E$.
For all bars $i=1,\ldots,10$ the vector $\gamma_i \in \R^8$ contains $-\cos(\alpha)$ in the components corresponding to a nodal displacement at one of the two end nodes of the bar $i$, where $\alpha$ is the angle between the respective nodal displacement axis and the bar axis.
This equation models force equilibrium and some other conditions.

The inequality $f^Tu \leq c$ bounds the compliance of the truss, i.e.\ the 
work caused by the force $f$.
Here, $c>0$ is a user-defined constant.

The box constraints $0 \leq a_i \leq \bar{a}$ ($i=1,\ldots,10$) ensure nonnegativity of the cross-sectional areas and allow a user-defined upper bound $\bar{a}>0$.
Additionally, one wants to impose bounds on the stress $\sigma_i(a,u)$ for the bar $i$, where
$$
   \sigma_i(a,u) = E \frac{\gamma_i^Tu}{l_i} \quad \forall i=1,
   \ldots,10,
$$
which is caused by the nodal displacement due to the force $f$.
This could be formulated as
$$
   \sigma_i(a,u)^2-\bar{\sigma}^2 \leq 0 \quad \forall i=1,\ldots,10,
$$
where $\bar{\sigma}>0$ is the user-defined threshold.
However, with this formulation, we would also bound the stress on those bars $i$ that do not appear in the final truss, i.e.\ those with $a_i=0$.
This is obviously not desirable, as it is unnecessarily restrictive.
We circumvent this by multiplying the inequalities above with $a_i$.
This eventually leads to the MPVC formulation given in \eqref{eq:tenbar}.

\begin{table}
    \centering
    \begin{tabular}{lccc}
    \toprule
    method & $V^*$ & constraint violation & iterations\\
    \midrule
    no regularization        & 8.0000 & 6.43708$e^{-11}$ & 1  \\
    global regularization    & 8.0000 & 1.5504$e^{-11}$ & 8\\
    local regularization     & 8.0000 & 1.8052$e^{-12}$&2 \\
    L-shaped regularization  & 8.0000 & 2.15399$e^{-11}$& 3 \\
    nonsmooth regularization & 8.1563 & 1.0415$e^{-15}$& 3\\
    \bottomrule
    \end{tabular}
    \caption{Results for ten-bar tuss}\label{tab:tenbar}
\end{table}

For this test, we chose the constants
\begin{equation*}
    E = 1, \quad
    c = 10, \quad
    \bar a = 100, \quad
    \bar \sigma = 1
\end{equation*}
and the initial point $a^0 = (1,\ldots,1) \in \R^{10}$ and $u^0 = K(a^0)^{-1}f$.

The algorithm terminated with the message \emph{Local minimizer found that satisfies the constraints} for all the methods, and computed the known optimal volume of $V^*=8.0000$ with negligible differences in the solutions for all methods except for the nonsmooth regularization.
The latter yielded a larger volume of $V^*=8.1563$ and a slightly different structure, see Figures \ref{fig:tenbarOpt} and \ref{fig:tenbarOptKad} for pictures of the corresponding trusses and Table \ref{tab:tenbar} for a comparison of the results for all the methods.

\subsection{Aerothermodynamic Problem}\label{ssec:aerothermodynamic}

During re-entry, aircrafts experience extreme thermal loads, which can be controlled by choosing an adequate trajectory using optimal control techniques; cf. e.g. \cite{Ger11, PMB13}.
Additionally, a liquid hydrogen active cooling system can be used.
However, this is only needed, when the heat loads exceed the radiative cooling abilities of the thermal protection materials; see \cite{AAS95, CPW09}.
Thus, a constraint on its activation must be imposed to avoid unnecessary fuel consumption.

\begin{figure}
    \centering
    \begin{tikzpicture}[scale = 0.9]
        \draw (0, 0) node[inner sep=0] {\includegraphics[width=4cm]{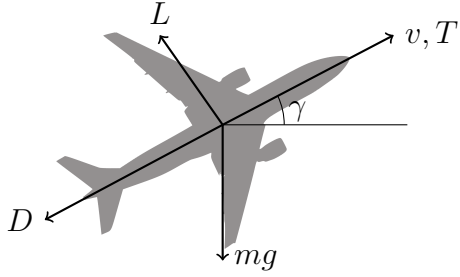}};
        \draw[->, thick] (0.3,0.3)--(2.8,1.62)  node[right]{$v, T$};
        \draw[->, thick] (0.3,0.3)--(-2.3,-1.1)  node[left]{$D$};;
        \draw[->, thick] (0.3,0.3)--(0.3,-1.7)  node[right]{$mg$};;
        \draw[->, thick] (0.3,0.3)--(-0.62,1.62)  node[above]{$L$};
        \draw[-] (0.3,0.3)--(3,0.3);
        \draw (1.2,0.3) arc(0:25:1);
        \draw (1.1,0.5) node[right]{$\gamma$};
    \end{tikzpicture}
    \caption{Aerodynamic forces on an aircraft}  \label{fig:aeroforces}
\end{figure}

Using optimal control terminology, the state variables of an aircraft are velocity $v = v(t)$, flight path angle $\gamma = \gamma (t)$, altitude $h = h(t)$ and total external heat load $Q_{T}=Q_T(t)$.
The control variables are angle of attack $C_L = C_L (t)$, thrust $T = T(t)$ and convected heat rate due to active cooling $\dot{Q_{c}} = \dot{Q_{c}}(t)$.
The forces defining the trajectory of the aircraft, pictured in Figure \ref{fig:aeroforces}, are aerodynamic lift $L$ and drag $D$, thrust $T$ and gravitational force $mg$.

The equations describing the trajectory and heat rate evolution at the stagnation point are
\begin{align*}
    \dot{v}&=  \frac{T - D(v,h;C_L)}{m} - g(h) \sin \gamma, \\
    \dot{\gamma}&= \frac{ L(v,h;C_L)}{m v} + \cos \gamma  \left( \frac{v}{r(h)} - \frac{g(h)}{v}  \right),  \\
    \dot{h}&= v \sin \gamma,  \\
    \dot{Q_{T}} &= K_{e} \sqrt{\frac{\rho (h)}{R_{n}} } v^{3} - \dot{Q_{c}},
\end{align*}
where $g(h)$ is the gravitational acceleration, $r(h)$ the distance to the center of the Earth, and $\rho(h)$ the air density for a given height $h$, $R_N$ is the nose radius, and $K_e$ is a constant.
Box constraints for the controls have to be imposed:
\begin{align*}
    0.01 &\leq C_L \leq 0.18326, \\
    0\ N &\leq T \leq 10^{7} N, \\
    0\ W/cm^{2} &\leq \dot{Q_{c}} \leq  0.5 \ W/cm^{2}.
\end{align*}

The landing condition translates to
\begin{align*}
    h(t_{f}) \leq 500\ m,
\end{align*}
and the cooling system activation constraint is given by
\begin{align*}
    (\dot{Q}_{rad.max}-\dot{Q_T}) \  \dot{Q_{c}} \leq 0,
\end{align*} 
for which we have chosen the value $\dot{Q}_{rad,max}=1.7\ W/cm^2$, which corresponds to reaching temperatures over 750 K.

Finally, the objective function to minimize is the final heat load $Q_T(t_f)$.
We now have a general optimal control problem of the form
\begin{equation}\label{eq:OpCoProb}
    \begin{array}{ll}
       \min & f(x_{f},u_{f}) \\
       \st  & \dot{x} = F(x,u),\\
       & h(x,u) = 0, \\
       & g(x,u) \leq 0, \\
       & H(x,u) \geq 0, \\
       & G(x,u)H(x,u)\leq 0, 
    \end{array}
\end{equation}
where $x = x(t) = (v(t),\gamma(t), h(t), Q_{T}(t))$ is the vector of state variables and $u = u(t) = (C_L(t), T(t), \dot{Q_{c}}(t))$ is the vector of control variables.

Discretizing in time and choosing a suitable integration scheme $\Phi (\cdot)$ with step $\delta$
\begin{align*}
    y_{i+1} = y_{i} + \delta \Phi (t_{i}, t_{i+1}, y_{i}, y_{i+1})
\end{align*}
where $t_{i+1} = t_{i} + \delta$ and $y_{i} \approx y(t_{i})$ for  $i=0,1,...,N-1$, we can approximate $(x(t_{i}), u(t_{i})) \approx (x_{i},u_{i})$, thus obtaining the following MPVC from problem \eqref{eq:OpCoProb}:
\begin{equation}\label{eq:DiscreteMPVC}
    \begin{array}{ll}
       \min & f(x_{N},u_{N}) \\
       \text{s.t.}    & x_{i+1} - x_{i} - \delta \Phi (t_{i}, t_{i+1}, x_{i}, x_{i+1}, u_{i}, u_{i+1}) = 0 \quad (i = 0,1,...,N-1),\\
       & h(x_i,u_i) = 0  \quad (i = 0,1,...,N-1),\\
       & g(x_i,u_i) \leq 0  \quad (i = 0,1,...,N-1),\\
       & H(x_i,u_i) \geq 0  \quad (i = 0,1,...,N-1),\\
       & G(x_i,u_i)H(x_i,u_i)\leq 0 \quad (i = 0,1,...,N-1).
    \end{array}
\end{equation}

We have chosen the implicit Euler method as integration scheme for this test with 30 time nodes and used a free final time transformation.
The initial values for the state and control variables were
\begin{equation*}
    v_0 = 0.2 \text{ km/s}, \quad
    \gamma_0 = 0 \text{ rad}, \quad
    h_0 = 12 \text{ km}, \quad
    Q_{T,0} = 0.0 \text{ J/cm}^2.
\end{equation*}

The obtained trajectories for each regularization method as well as the heat load evolution are depicted in Figure \ref{fig:trajectory}.
For all five methods, the algorithm terminated with the message \emph{Local minimizer found that satisfies the constraints}.
The lowest final heat loads was obtained by the L-shaped regularization.
The direct approach without regularization, the global regularization and the nonsmooth regularization resulted in similar final heat loads and the local regularization found a slightly higher final heat load.

\begin{figure}
    \centering
    \includegraphics[width=\linewidth]{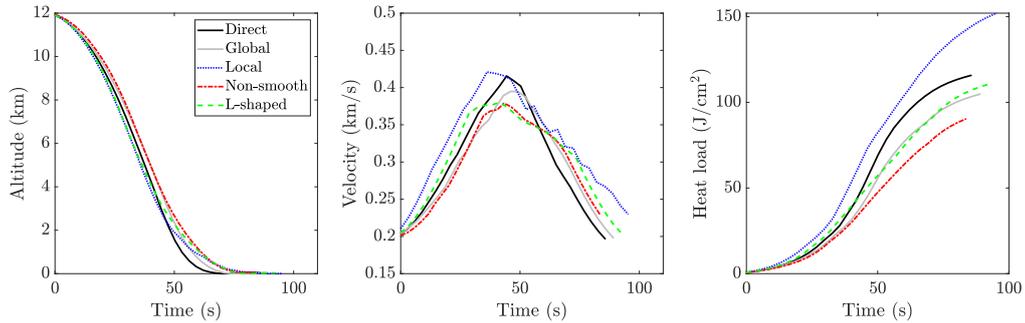}
    \caption{Solutions for the Aerothermodynamic problem \eqref{eq:OpCoProb}}\label{fig:trajectory}
\end{figure}

\subsection{Summary of the Numerical Tests}\label{ssec:summaryNumerics}

The previous examples show that all four regularization methods can successfully be applied to both truss design and optimal control problems with realistic and complex modelling.
The computed solutions are at least as good as the result of the direct approach without regularization.
In case of the aerothermodynamic problem \eqref{eq:OpCoProb}, the L-shaped regularization was even able to decrease the final heat load by approx. 20\% compared to the direct approach, and the academic example \eqref{eq:academicTruss} indicates that the regularization methods are less dependent on the initial point than the direct approach.

Discriminating between the four regularization methods based on these numerical results is subtle:
The global regularization generated satisfactory solutions for all text examples and was the most robust with respect to the initial point in the academic example \eqref{eq:academicTruss}.
However, it also had the highest iterations numbers in this example.
The local regularization was more sensitive to the choice of initial points in the academic example \eqref{eq:academicTruss} than the other regularizations and generated the worst solution for the aerothermodynamic problem \eqref{eq:OpCoProb}.
The L-shaped regularization found the global optimum in the academic example \eqref{eq:academicTruss} more often than all other approaches and the best solution for the aerothermodynamic problem \eqref{eq:OpCoProb}.
The nonsmooth regularization finally was the only approach to terminate in a suboptimal solution in the ten-bar truss problem \eqref{eq:tenbar}.
\section{Final Remarks}\label{sec:final}

In this paper, we analyzed four one-parameter regularization schemes for MPVCs with a focus on their theoretical convergence properties in the exact and inexact case as well as on the regularity of the resulting feasible sets.
The central observation here is that the simplest approach, the global regularization has the worst theoretical properties in the exact setting but is able to retain those in the inexact setting.
The other three, more involved regularization schemes have better theoretical properties, if one is able to compute KKT points of the regularized problems.
But when only $\eps$-stationary points can be computed, their properties are worse than those of the global regularization.

A second result of the theoretical analysis is that the transfer of these regularization schemes from MPCCs to MPVCs sometimes requires additional assumptions not present in the MPCC setting.
These are needed to cope with the sign constraint $\eta^G_i \geq 0$ for $i \in I_{00}$, which is required already for weak stationarity.

In addition to the theoretical analysis, we also performed numerical tests based on three examples from truss design and optimal control.
These examples illustrate that the use of a regularization is beneficial compared to the direct application of an NLP solver to the MPVC.
The numerical results also indicate that the global regularization is quite robust, but it can be slow and terminate in suboptimal solutions.
Contrary to this, the L-shaped regularization seems to be faster and find better solutions, but might be less robust with respect to the initial point.
Those observations would also fit well with the theoretical analysis.
However, in order to be able to make solid claims in this regard, a more extensive numerical study is necessary.
Unfortunately, at the moment no suitable collection of MPVC test problems is available.
\section*{Funding}
The work of T. Hoheisel was supported by NSERC Discovery Grant RGPIN-2017-04035.
The work of B. Pablos was supported by the Bavarian Research Alliance and Munich Aerospace.
The work of A. Pooladian was supported by NSERC CGS-M and Lorne Trottier Accelerator Fellowship.
The work of A. Schwartz was supported by the Excellence Initiative of the German Federal and State Governments and the Graduate School of Computational Engineering at Technische Universit\"at Darmstadt.
The work of L. Steverango was supported by an ISM Undergraduate Summer Research Scholarship.

\printbibliography

\end{document}